\documentclass[review,onefignum,onetabnum]{siamart171218}

\usepackage{amssymb}
\usepackage{titlesec}
\usepackage{amsmath,amscd}
\usepackage{amsfonts}
\usepackage{graphicx}
\usepackage{ulem}
\usepackage{subfigure}
\usepackage{amssymb}
\usepackage{float}
\usepackage{multirow}
\usepackage{diagbox}
\usepackage{mathrsfs}
\usepackage{indentfirst}
\usepackage{geometry}
\usepackage{xcolor}
\usepackage{amssymb}
\usepackage{enumitem}
\usepackage{tikz}
\usepackage{etoolbox}
\usepackage{geometry}
\usepackage{fancyhdr}
\usepackage{lastpage}
\usepackage{layout}
\usepackage{titlesec}
\newtheorem{assumption}{Assumption}[section]

\newcommand{\circled}[2][]{\tikz[baseline=(char.base)]
    {\node[shape = circle, draw, inner sep = 1pt]
    (char) {\phantom{\ifblank{#1}{#2}{#1}}};%
    \node at (char.center) {\makebox[0pt][c]{#2}};}}
\robustify{\circled}

\usepackage{lipsum}
\usepackage{amsfonts}
\usepackage{graphicx}
\usepackage{epstopdf}
\usepackage{algorithmic}
\ifpdf
  \DeclareGraphicsExtensions{.eps,.pdf,.png,.jpg}
\else
  \DeclareGraphicsExtensions{.eps}
\fi

\newsiamremark{remark}{Remark}
\newsiamremark{hypothesis}{Hypothesis}
\crefname{hypothesis}{Hypothesis}{Hypotheses}
\newsiamthm{claim}{Claim}

\headers{Stochastic primal dual fixed point method for composite optimization}{YaNan Zhu and~ Xiaoqun Zhang}

\title{Stochastic primal dual fixed point method for composite optimization\thanks{Submitted to the editors DATE.}}

\author{YaNan Zhu and Xiaoqun Zhang}

\usepackage{amsopn}

\ifpdf
\hypersetup{
  pdftitle={Stochastic primal dual fixed point method for composite minimization},
  pdfauthor={Yanan Zhu and Xiaoqun Zhang}
}
\fi

\begin{document}

\maketitle

\begin{abstract}
In this paper we propose a stochastic  primal dual fixed point method (SPDFP) for solving the sum of two proper lower semi-continuous convex function and one of which is composite. The method is based on the primal dual fixed point method (PDFP) proposed in \cite{PDFP} that does not require subproblem solving. Under some mild condition, the convergence is established based on two sets of assumptions: bounded and unbounded gradients and the convergence rate of the expected error of iterate is of the order $\mathcal{O}(k^{-\alpha})$ where $k$ is iteration number and $\alpha \in (0,1]$. Finally, numerical examples on graphic Lasso and logistic regressions are given to demonstrate the effectiveness of the proposed algorithm.
\end{abstract}



\section{Introduction.}
This paper is devoted to discussing a stochastic algorithm based on the primal dual fixed point method (PDFP) \cite{PDFP} for minimizing the sum of two proper semi-continuous convex functions, one of which is a composite, i.e.,
\begin{equation}\label{pb1}
x^* = \underset{x \in \mathbb{R}^d}{\arg\min~} (f_1 \circ B) (x) + f_2(x),
\end{equation}
where $f_1(x)$ may not be differentiable and $B : \mathbb{R}^{d} \rightarrow \mathbb{R}^{m}$ is a linear transform. The function $f_2(x)$ is proper convex lower semi-continous on $\mathbb{R}^d$ with the following form:
\begin{equation}\label{form1}
f_2(x) = \frac{1}{n} \sum_{i = 1}^{n}\phi_i(x),
\end{equation}
where $\phi_i(x), i = 1,\cdots,n$ are smooth convex functions.  \\
\indent Many problems in machine learning can be formulated as in \textbf{(\ref{pb1})}. For example, the generalized lasso \cite{Generalizedlasso} is given as follows:
\begin{equation}\label{pb11}
x^* = \underset{x \in \mathbb{R}^d}{\arg\min~}
\frac{1}{n}\sum_{i = 1}^{n}\phi_i(x)+ \mu \lVert Bx \rVert_1,
\end{equation}
where $x \in \mathbb{R}^d$, $\mu > 0$ is the regularization parameter and $B$ the penalty matrix specifying a certain sparsity pattern of $x$. The function $\phi_i(x)$ can be square loss $\phi_i(x) = \frac{1}{2}(a_i^Tx - b_i)^2$, logistic loss $\phi_i(x) = \log(1 + \exp(-b_ia_i^Tx))$, and hinge loss $\phi_i(x) = \max\{ 0,1 - b_i a_i^Tx \}$, where $a_i \in \mathbb{R}^d$ denotes the $i$th sample and $b_i \in \mathbb{R}$ denotes the label of the sample.\\

\subsection{Motivation.}
The problem dates back to the structural risk minimization principle in \cite{Vapnik}, in which the goal of statistical learning is to minimize the regularized expected risk function:
\begin{equation*}
F(x) = \mathbb{E}_{\xi}(L(x,\xi)) + R(x),
\end{equation*}
where $L(x,\xi)$ is the loss when applying the prediction rule $x$ on $\xi$ and $R(x)$ is the regularizer. Rather than computing the expectation of the loss, one can use a set of training samples to minimize the regularized empirical risk:
\begin{equation*}
F_{emp}(x) = \frac{1}{n}\sum_{i = 1}^{n}L(x,\xi_i) + R(x).	
\end{equation*}
If the number of samples $n$ is large enough, the regularized expected risk $F(x)$ and the regularized empirical risk
$F_{emp}(x)$ are close with high probability.
Since the number of training samples is large, computing the gradient of the first term of the $F_{emp}(x)$ will be time consuming, which means it will be inefficient when we directly use many formal algorithms to minimize $F_{emp}(x)$. Thus, instead of using all the samples to compute the full gradient, the stochastic algorithms use one sample or a small portion of the samples to compute a noisy gradient in each iteration to reduce the complexity of the algorithms.\\

\subsection{Recent work.}
If the linear transform $B$ is the identity, the problem can be rewritten as
\begin{equation}\label{pb2}
x^* = \underset{x \in \mathbb{R}^d}{\arg\min~} f_1 (x) + f_2(x),
\end{equation}
and there are many algorithms that can solve it. For deterministic (or batch) algorithms, e.g., the proximal gradient descent (PGD) method (also known as proximal forward backward splitting (PFBS)]  \cite{PFBS} and its acceleration versions \cite{FISTA,Nesterov,FastPPT}. If $f_2(x)$ has the form \textbf{(\ref{form1})}, many researchers study the stochastic version of PGD algorithms, e.g., stochastic proximal gradient (SPG) descent \cite{converPGD,FOBOS}. The difference between PGD and SPG descent is that PGD uses all the samples to compute the full gradient in each iteration, while SPG descent uses one or a small portion of the samples to compute a noisy gradient. Owing to variance caused by random sampling, SPG descent uses diminishing step size, which leads to a sub-linear convergence rate. To accelerate the convergence rate
and use a larger step size, many researchers apply variance-reduction techniques to these algorithms, e.g., the proximal stochastic dual coordinate ascent (Prox-SDCA) method \cite{ProxSDCA} and its acceleration \cite{ACProxSDCA} and the proximal stochastic gradient method with variance reduction (Prox-SVRG) \cite{ProxSVRG}. \\
\indent If the linear transform $B \not = I$, PGD and SPG descent will have to solve $\mathrm{Prox}_{f_1 \circ B}(\cdot)$, which is not easy in many problems. To deal with it, many deterministic and stochastic algorithms were designed, e.g., the split Bregman method \cite{DR1,DR2}, the alternating direction of multipliers method (ADMM) \cite{ADMM1,ADMM2}, and the fixed-point method based on proximity operator ($\mathrm{FP^2O}$) \cite{FP2O}. Inspired by $\mathrm{FP^2O}$, the authors in \cite{PDFP} used just one step for the sub-problem of $\mathrm{FP^2O}$ and proposed the primal dual fixed-point (PDFP) algorithms. Noting the simplicity of the PDFP approach for solving the problem \textbf{(\ref{pb1})}, we study the stochastic version of the PDFP algorithms and apply it to solve machine-learning problems.\\
\indent Among the stochastic algorithms used to solve \textbf{(\ref{pb1})}, the most popular stochastic algorithms are the stochastic versions of ADMM. There are many types of the stochastic ADMMs, e.g., stochastic ADMM (STOC-ADMM) \cite{STOCADMM}, the convergence rate of which is $\mathcal{O}(1/\sqrt{k})$ for general convex functions and $\mathcal{O}{(\log(k)/k)}$ for strongly convex functions, where $k$ is the iteration number. Based on regularized dual averaging (RDA) \cite{RDA} and the online proximal gradient (OPG) \cite{OPG}, Suzuki proposed the RDA-ADMM and OPG-ADMM \cite{RDAOPGADMM}, the convergence rates of which are the same for general and strongly convex objective functions as STOC-ADMM. To accelerate the algorithms, several researchers have utilized variance-reduction techniques in the stochastic ADMMs. For example, Zhong-Kwok proposed the stochastic averaged gradient ADMM (SA-ADMM) \cite{SAGADMM} by combining the stochastic averaged gradient \cite{SAG} and ADMM. The SA-ADMM can achieve a $\mathcal{O}{(1/k)}$ convergence rate for general convex objective functions, while the rate for strongly convex functions is unknown. SA-ADMM requires extra memory to store historical gradients to approximate the full gradient, which is not scalable in term of storage. Thus, Zhao et al. proposed the scalable ADMM (SCAS-ADMM) \cite{SCASADMM} with a $\mathcal{O}(1/k)$ convergence rate for general and strongly convex objective functions. Although SA-ADMM and SCAS-ADMM use the variance-reduction technique, they both have a sublinear convergence rate, which defeats the original purpose of using variance reduction. Thus, Suzuki \cite{SDCAADMM} and Zheng-Kwok \cite{SVRGADMM} put forward the stochastic dual coordinate ascent ADMM (SDCA-ADMM) and stochastic variance-reduced ADMM (SVRG-ADMM) methods, respectively. Both algorithms can achieve a linear convergence rate for strongly convex functions. Different from SDCA-ADMM, SVRG-ADMM is more scalable in terms of storage.
Table \textbf{\ref{SADMMT}} summarizes the convergence rate and storage requirements of different stochastic ADMM algorithms.\\
In this paper. we propose a stochastic version of PDFP. Under the boundness and non-boundness gradients of function $f_2(x)$, we prove both the convergence and convergence rate of SPDFP. Compared with the stochastic version of ADMM, the computation of SPDFP only involves matrix vector multiplication and proximal operation, which is easy for many problems. Therefore, SPDFP can serve as an alternative to solve many machine-learning problems. In addition, numerical results of SPDFP on fussed lasso, graph guide support vector machine (SVM) and graph guide logistic regression on real data show some advantages compared to other state-of-the-art methods.

\begin{table}
\caption{Convergence rate of stochastic ADMMs}
\begin{center}
\begin{tabular}{c c c c c c c c}
\hline
\hline
\multicolumn{2}{c}{Algorithms}
& \multicolumn{2}{c}{general convex}
& \multicolumn{2}{c}{strongly convex}
& \multicolumn{2}{c}{storage}
\\
\cline{1-8}
 \multicolumn{2}{c}{\small{STOC-ADMM}  \cite{SVRGADMM}}
& \multicolumn{2}{c}{$\mathcal{O}(1/\sqrt{k})$}
& \multicolumn{2}{c}{$\mathcal{O}(\log(k)/k)$}
& \multicolumn{2}{c}{$\mathcal{O}(md + d^2)$} \\
 \multicolumn{2}{c}{\small{RDA-ADMM} \cite{RDAOPGADMM}}
& \multicolumn{2}{c}{$\mathcal{O}(1/\sqrt{k})$}
& \multicolumn{2}{c}{$\mathcal{O}(\log(k)/k)$}
& \multicolumn{2}{c}{$\mathcal{O}(md)$}\\
 \multicolumn{2}{c}{\small{OPG-ADMM} \cite{RDAOPGADMM}}
& \multicolumn{2}{c}{$\mathcal{O}(1/\sqrt{k})$}
& \multicolumn{2}{c}{$\mathcal{O}(\log(k)/k)$}
& \multicolumn{2}{c}{$\mathcal{O}(md)$}\\
\cline{1-8}
 \multicolumn{2}{c}{\small{SAG-ADMM} \cite{SAGADMM}}
& \multicolumn{2}{c}{$\mathcal{O}(1/k)$}
& \multicolumn{2}{c}{$unknown$}
& \multicolumn{2}{c}{$\mathcal{O}(nd + md)$}\\
 \multicolumn{2}{c}{\small{SCAS-ADMM} \cite{SCASADMM}}
& \multicolumn{2}{c}{$\mathcal{O}(1/k)$}
& \multicolumn{2}{c}{$\mathcal{O}(1/k)$}
& \multicolumn{2}{c}{$\mathcal{O}(md)$}\\
 \cline{1-8}
 \multicolumn{2}{c}{\small{SDCA-ADMM} \cite{SDCAADMM}}
& \multicolumn{2}{c}{$unknown$}
& \multicolumn{2}{c}{$linear$}
& \multicolumn{2}{c}{$\mathcal{O}(n + md)$}\\
 \multicolumn{2}{c}{\small{SVRG-ADMM} \cite{SVRGADMM}}
& \multicolumn{2}{c}{$\mathcal{O}(1/k)$}
& \multicolumn{2}{c}{$linear$}
& \multicolumn{2}{c}{$\mathcal{O}(md)$}\\
\cline{1-8}
\end{tabular}
\label{SADMMT}
\end{center}
\end{table}

\subsection{Organization of this paper.}
\noindent The paper is organized as follows.
In Section 2, we present notations and lemmas that are used throughout the paper.
In Section 3, we introduce SPDFP Algorithm 1 and its equivalent form Algorithm 2. Algorithm 1 will be used in convergence analysis, while Algorithm 2 will be used in implementation for numerical stability concerns.
We then provide the convergence analysis of the algorithms. Based on different assumptions, we give different convergence results of the algorithms.
Finally, several numerical examples are given to show the effectiveness of the algorithms.

\section{Preliminaries.}
In this section, we present notation and lemmas used througnout the paper.
\begin{definition}\label{proxf}
The operator $\mathrm{Prox}_{f}(\cdot)$ is defined by
\begin{equation*}
\begin{aligned}
\mathrm{Prox}_{f}(y)~ ( \mathbb{R}^d \rightarrow \mathbb{R}^d ) : ~
& y \rightarrow \underset{x \in \mathbb{R}^d}{\arg\min}\big\{ f(x) + \frac{1}{2} \lVert x - y \rVert_2^2 \big\}.
\end{aligned}
\end{equation*}
\end{definition}

\begin{definition}
Letting $f$ be a proper convex lower semi-continuous function on $\mathcal{X}$, then the sub-differential of $f$ is a set-valued operator $\partial f:\mathcal{X} \to 2^{\mathcal{X}}$ defined by
\begin{equation*}
\partial f(x) = \{ s \in \mathcal{X} | f(y)\geq f(x) + \langle s,y - x \rangle ~for ~ all~ y \in \mathcal{X}\}.	
\end{equation*}
\end{definition}
\begin{definition}\label{lm1}
An operator $T : \mathbb{R}^d \rightarrow \mathbb{R}^d$ is firmly non-expansive if and only if it satisfies
\begin{equation*}
\lVert Tx - Ty \rVert_2^2 \leq \big< Tx - Ty,x - y \big>
\end{equation*}
for all $(x,y) \in  \mathbb{R}^d \times \mathbb{R}^d$. It can be verified that $\mathrm{Prox}_f$  and $I - \mathrm{Prox}_f$ are firmly non-expansive.
\end{definition}

\begin{lemma}\label{lm2}
If $f$ is convex and has $\frac{1}{\beta}$-Lipschitz continuous gradients, then
\begin{equation*}
\big< \nabla f(x) - \nabla f(y),x - y \big> \geq \beta \lVert \nabla f(x) - \nabla f(y)\rVert_2^2
\end{equation*}
for all $x,y \in \mathbb{R}^m$. \\
If $f$ is $\nu$-strongly convex, then
\begin{equation*}
\big< \nabla f(x) - \nabla f(y),x - y \big> \geq \nu \lVert x - y\rVert_2^2.
\end{equation*}
\end{lemma}

\begin{definition}\label{feq}
\cite{lemma,polyak,converPGD} Define a family of functions $(\varphi_c)_{c \in \mathbb{R}}$ as follows:
\begin{equation*}
\varphi_c: (0, +\infty) \rightarrow \mathbb{R}: t \rightarrow
\left \{
\begin{aligned}
(t^c - 1)/c\ \qquad &if ~c \not = 0 \\
\log(t) \qquad & if ~c = 0.
\end{aligned}
\right.
\end{equation*}
\end{definition}

The following lemma is important in the proof of the convergence rate of SPDFP. It can also be found in Lemma 4.4 of \cite{converPGD}.
\begin{lemma}\label{lm4}
\cite{lemma,polyak,converPGD} Let $\alpha$ be in $(0,1)$, and let $c$ and $\tau$ be in $(0,\infty)$; $(\eta_k)_{k \in \mathbb{N}^*}$ is a strictly positive sequence defined by  $\eta_k = \frac{c}{k^{\alpha}}$ and $k \in \mathbb{N}^*$. $(s_k)_{k \in \mathbb{N}^*}$ is a sequence satisfying
\begin{equation*}
0 \leq s_{k + 1} \leq (1 - \eta_k)s_k + \tau \eta_k^2 \qquad \forall k \in \mathbb{N}^*.
\end{equation*}
Let $k_0$ be the smallest integer such that $\eta_{k_0} \leq 1$. Then, for every $k \geq 2k_0$ the following estimate holds:
\begin{equation*}
s_{k + 1} \leq \left\{
\begin{aligned}
& \Big( \tau c^2 \varphi_{1 - 2\alpha}(k) + s_{k_0}\exp \Big( \frac{ck_0^{1 - \alpha}}{1 - \alpha} \Big)\Big) \exp\Big( \frac{-c(1 - 2^{\alpha - 1})(k + 1)^{1 - \alpha}}{1 - \alpha} \Big) + \frac{\tau2^{\alpha}c}{(k - 2)^{\alpha}},\\
& \qquad\qquad\qquad\qquad\qquad\qquad\qquad\qquad\qquad\qquad if ~ \alpha \in (0,1). \\
& s_{k_0}\Big( \frac{k_0}{k + 1} \Big)^c + \frac{\tau c^2}{(k + 1)^c}(1 + \frac{1}{k_0} )^c \varphi_{c - 1}(k) \qquad if ~ \alpha = 1.
\end{aligned}
\right.
\end{equation*}
\end{lemma}

\begin{remark}\label{rk1}
From Lemma \textbf{\ref{lm4}}, it can be seen that
\begin{equation*}
\begin{aligned}
s_k = \left \{
\begin{aligned}
& \mathcal{O}(k^{-\alpha}) \qquad \qquad \qquad if ~\alpha \in (0,1) \\
& \mathcal{O}(k^{-1}) + \mathcal{O}\Big(\frac{\log(k)}{k}\Big)\qquad if ~ \alpha = 1,\qquad c  = 1.  \\
& \mathcal{O}(k^{-c}) + \mathcal{O}(k^{-1})\qquad if ~ \alpha = 1,\qquad c \not = 1.
\end{aligned}
\right.
\end{aligned}
\end{equation*}
\end{remark}

\section{Algorithms.}
In this section, we present the SPDFP and its variant. First, we present the fixed-point formulation for the solution of \textbf{(\ref{pb1})}, which can be found in \textbf{Theorem 3.1} of \cite{PDFP}.
\begin{theorem}
(\cite{PDFP},Theorem 3.1)~Let $\lambda$ and $\gamma$ be two positive numbers. Supposing that $x^*$ is a solution of (\textbf{\ref{pb1}}), then there exists $v^* \in \mathbb{R}^m$ such that
\begin{equation}\label{FP}
\left\{
\begin{aligned}
v^* & = (I - \mathrm{Prox}_{\frac{\gamma}{\lambda}f_1})(B(x^* - \gamma \nabla f_2(x) + (I - \lambda BB^T)v^*) \\
& = \tilde{T}_0(x^*,v^*) \\
x^* & = x^* - \gamma \nabla f_2(x) - \lambda B^T \tilde{T}_0(x^*,v^*)
\end{aligned}
\right.	
\end{equation}
\end{theorem}
The PDFP is given as
\hspace*{\fill} \\
\begin{center}
\fbox{
\shortstack[l]{
	\textbf{Algorithm: Primal dual fixed-point method} \\
	\rule{300pt}{2pt}\\
    step 1:set $x_1 \in \mathbb{R}^d,v_1 \in \mathbb{R}^{m}$ and choose proper $\gamma>0,\lambda > 0$, then \\
    step 2:for $k = 1,2,\cdots$ \\
    \qquad\qquad $x_{k + \frac{1}{2}} = x_k - \gamma\nabla f_2(x_k)$ \\
    \qquad\qquad$v_{k + 1} = \big(I - \mathrm{Prox}_{\frac{\gamma}{\lambda}f_1}\big)\big(  Bx_{k + \frac{1}{2}} + (I - \lambda B B^T)v_k \big)$ \\
    \qquad\qquad $x_{k + 1} = x_{k + \frac{1}{2}} - \lambda B^T v_{k+1}$\\
    until the stopping criterion is satisfied.
}
}
\end{center}
\hspace*{\fill} \\
From \textbf{(\ref{FP})}, it can be seen that the PDFP is actually a fixed-point iteration. In addition, PDFP uses the full gradient and constant step size in each step, which is not applicable for many large-scale problems. Generally, the  stochastic version of PDFP will use a stochastic gradient instead of a full gradient and diminishing step size to reduce the variance caused by random sampling. \\
We now give the notation for the stochastic gradient.
Recalling the computation of the full gradient,
\begin{equation*}\label{GD}
\nabla f_2(x) = \frac{1}{n}\sum_{j = 1}^{n} \nabla \phi_j(x)
\end{equation*}
the computation of the stochastic gradient is given as follows:
\begin{itemize}
	\item Given batch size $p$, the set $\{ 1 ,2,\cdots,n\}$ is divided into $\frac{n}{p}$ non-overlapping batches, where the number of entries of each batch is $p$.
	\item Then, the batch $i$ is chosen with probability $\frac{p}{n}$ and the stochastic gradient of $f_2$ at $x$ [denoted $\nabla f_2^{[i]}(x)$] is given by
\begin{equation*}\label{SGD}
\nabla f_2^{[i]}(x) = \frac{1}{p}\sum_{j = (i - 1)p +1}^{ip} \nabla \phi_j(x).
\end{equation*}
\end{itemize}
It can be verified that the expectation of the stochastic gradient at $x$ is exactly the full gradient, i.e., $\mathbb{E}(\nabla f_2^{[i]}(x)) = \nabla f_2(x)$, which means that the stochastic gradient is an unbiased estimate of the full gradient.
\newpage
For a stochastic setting, the fixed-point formulation in (\textbf{\ref{FP}}) should be modified, and we give the following lemma.
\begin{lemma}\label{lm3}
Let $\lambda > 0, \gamma_k > 0$ be two positive numbers. Supposing that $x^*$ is a solution of (\ref{pb1}), then there exists $v^*  \in \partial f_1(Bx^*)$ such that
\begin{equation}\label{eq1}
\left\{
\begin{aligned}
v^* & = (I - \mathrm{Prox}_{h_k})\big(\frac{\lambda}{\gamma_k} B(x^* - \gamma_k\nabla f_2(x^*))
+ (I - \lambda BB^T)v^* \big)     \\
& = T_0^{(k)}(x^*,v^*) \\
x^* & = x^* - \gamma_k \nabla f_2(x^*) - \gamma_k B^T\circ T_0^{(k)}(x^*,v^*)
\end{aligned}
\right.,
\end{equation}
where $h_k(x) = \frac{\lambda}{\gamma_k}f_1(\frac{\gamma_k}{\lambda}x)$. Conversely, if $x^*, v^*$ satisfy equation \textbf{(\ref{eq1})} then $x^*$ is the solution of the problem \textbf{(\ref{pb1})}. \\
\end{lemma}

\begin{proof}
See Appendix.
\end{proof}

\noindent Using Lemma \textbf{\ref{lm3}}, considering the discussion above, and noting the fact that
\begin{equation}
\begin{aligned}
& v^*  = \big(I - \mathrm{Prox}_{h_k}\big)\big( \frac{\lambda}{\gamma_k}B(x^* - \gamma_k \nabla f_2(x^*)) + (I - \lambda B B^T)v^* \big) \\
& \Leftrightarrow
v^*  = \frac{\lambda}{\gamma_k}\big(I - \mathrm{Prox}_{\frac{\gamma_k}{\lambda}f_1}\big)\big(B(x^* - \gamma_k \nabla f_2(x^*)) + (I - \lambda B B^T)\frac{\gamma_k}{\lambda}v^* \big) ,
\end{aligned} 	
\end{equation} we introduce the SPDFP:

\hspace*{\fill} 
\begin{center}
\fbox{
\setlength{\fboxsep}{0.1cm}
\shortstack[l]{
	\textbf{Algorithm 1: Stochastic primal dual fixed-point method} \\
    \rule{300pt}{2pt}\\
    Step 1: Set $x_1 \in \mathbb{R}^d,v_1 \in \mathbb{R}^{m}$ and choose proper $c,\lambda$, and $\alpha \in (0,1]$\\
    Step 2: For $k = 1,2,\cdots$ \\
    \qquad\qquad $\gamma_k = \frac{c}{k^{\alpha}}$ \\
    \qquad\qquad ~choose $i_k$  randomly from $1,2,\cdots,n/p$ with probability $\frac{p}{n}$.\\
    \qquad\qquad $x_{k + \frac{1}{2}} = x_k - \gamma_k \nabla f_2^{[i_k]}(x_k)$ \\
    \qquad\qquad$v_{k + 1} = \frac{\lambda}{\gamma_k}\big(I - \mathrm{Prox}_{\frac{\gamma_k}{\lambda}f_1}\big)\big(  Bx_{k + \frac{1}{2}} + (I - \lambda B B^T)\frac{\gamma_k}{\lambda}v_k \big)$ \\
    \qquad\qquad $x_{k + 1} = x_{k + \frac{1}{2}} - \gamma_k B^T v_{k+1}$\\
    until the stopping criterion is satisfied.
}
}
\end{center}
\newpage
\hspace*{\fill} \\
\noindent Letting $v_{k + 1} := \frac{\gamma_k}{\lambda}v_{k + 1}$, we have the following equivalent form:
\begin{center}
\fbox{
\shortstack[l]{
  \textbf{Algorithm 2: Stochastic primal dual fixed-point method} \\
  \rule{300pt}{2pt}\\
  Step 1: ~Set $x_1 \in \mathbb{R}^d,v_1 \in \mathbb{R}^{m}$ and choose proper $c$ and $\alpha \in (0,1)$\\
    Step 2: \quad Let $k = 1,\gamma_1 = c$\\
    \qquad\qquad choose $i_1$ randomly from $1,2,\cdots,n/p$ with probability $\frac{p}{n}$\\
    \qquad\qquad$x_{1 + \frac{1}{2}} = x_1 - \gamma_1 \nabla f_2^{[i_1]}(x_1)$   \\
    \qquad\qquad$v_{2} = \big(I - \mathrm{Prox}_{\frac{\gamma_1}{\lambda}f_1}\big)\big(  Bx_{1 + \frac{1}{2}} + (I - \lambda B B^T)\frac{\gamma_1}{\lambda}v_1 \big)$ \\
    \qquad\qquad $x_{2} = x_{1 + \frac{1}{2}} - \lambda B^T v_{2}$\\
    Step 3: ~For $k = 2,3,\cdots$ \\
    \qquad\qquad $\gamma_k = \frac{c}{k^{\alpha}}$ \\
    \qquad\qquad ~choose $i_k$ randomly from $1,2,\cdots,n/p$ with probability $\frac{p}{n}$\\
    \qquad\qquad $x_{k + \frac{1}{2}} = x_k - \gamma_k \nabla f_2^{[i_k]}(x_k)$\\
    \qquad\qquad$v_{k + 1} = \big(I - \mathrm{Prox}_{\frac{\gamma_k}{\lambda}f_1}\big)\big( Bx_{k + \frac{1}{2}} + \big(\frac{k - 1}{k}\big)^{\alpha}(I - \lambda B B^T)v_k \big)$ \\
    \qquad\qquad $x_{k + 1} = x_{k + \frac{1}{2}} - \lambda B^T v_{k+1}$\\
	until the stop criterion is satisfied.
}
}
\end{center}

\begin{remark}
Algorithms 1 and 2 are equivalent by changing the notation. In the following, we use Algorithm 1 for convergence analysis and Algorithm 2 for the numerical simulations.
\end{remark}

It can be seen that the differences between PDFP and SPDFP (Algorithm 2) lie in step size and the computation of gradient: PDFP uses a constant step size and full gradient, while SPDFP uses diminishing step size and stochastic gradient. The reason for using diminishing step size in SPDFP is to reduce the variance caused by random sampling. Moreover, there is a difference in the update of the variable $v_k$, i.e., in SPDFP there is a factor $(\frac{k - 1}{k})^{\alpha}$.
\section{Convergence analysis.}
In this section, we present the convergence results of SPDFP. Under the assumption of bounded and non-boundness gradients of function $f_2(x)$, we give both the convergence and convergence rate of SPDFP. In addition, Table \textbf{\ref{cgsum}} is given to summarize the convergence results.

Before we show the convergence, we present the basic results on the sequence.
\begin{lemma}\label{cglm1}
Choosing $0 \leq \lambda \leq 1/\rho_{\max}(BB^T)$ and letting $(x_k,v_k)$ be the iteration of Algorithm 1 and $(x^*,v^*)$ as in Lemma \textbf{\ref{lm3}}, we then have the following estimate:
\begin{equation}\label{cglmeq1}
\begin{aligned}
& \mathbb{E}^{(k + 1)}\Big(\lVert x_{k + 1} - x^* \rVert_2^2 + \frac{\gamma_{k + 1}^2}{\lambda}\lVert v_{k + 1} - v^* \rVert_2^2\Big) \\
& \leq \mathbb{E}^{(k)}\big(\lVert x_k - x^* \rVert_2^2\big) + \frac{\gamma_{k}^2}{\lambda}(1 - \lambda \rho_{\min}(B B^T)) \mathbb{E}^{(k)}( \lVert v_k - v^* \rVert_2^2 ) \\
& - 2\gamma_k \mathbb{E}^{(k)}\big< \nabla f_2(x_k) - \nabla f_2(x^*), x_k - x^*\big> + \gamma_k^2 \mathbb{E}^{(k + 1)}(\lVert \nabla f_2^{[i_k]}(x_k) - \nabla f_2(x^*) \rVert_2^2),
\end{aligned}
\end{equation}
where $\mathbb{E}^{(k)}(\cdot)$ denotes the expectation up to the $k$-th iteration and $\rho_{max}(BB^T),\rho_{min}(BB^T)$ denote the maximum and minimum eigenvalues of matrix $BB^T$, respectively.
\end{lemma}
\begin{proof}
See Appendix.\qquad
\end{proof}

The inequality \textbf{(\ref{cglmeq1})} is essential in the following convergence analysis. Here and in what follows we use the notation
\begin{equation}\label{notation1}
a_k = \mathbb{E}^{(k)}\Big(\lVert x_{k} - x^* \rVert_2^2 + \frac{\gamma_{k}^2}{\lambda}\lVert v_{k} - v^* \rVert_2^2\Big),	
\end{equation}
and then the inequality \textbf{(\ref{cglmeq1})} in Lemma \textbf{\ref{cglmeq1}} can be rewritten as
\begin{equation}\label{cglmeq11}
\begin{aligned}
a_{k + 1} & \leq \mathbb{E}^{(k)}\big(\lVert x_k - x^* \rVert_2^2\big) + \frac{\gamma_{k}^2}{\lambda}(1 - \lambda \rho_{\min}(B B^T)) \mathbb{E}^{(k)}( \lVert v_k - v^* \rVert_2^2 ) \\
& - 2\gamma_k \mathbb{E}^{(k)}\big< \nabla f_2(x_k) - \nabla f_2(x^*), x_k - x^*\big> + \gamma_k^2 \mathbb{E}^{(k + 1)}(\lVert \nabla f_2^{[i_k]}(x_k) - \nabla f_2(x^*) \rVert_2^2).
\end{aligned}
\end{equation}
\subsection{Bounded gradient.}
Now we establish the convergence of SPDFP based on the uniform boundedness of the gradient of $f_2(x)$.
\begin{theorem}\label{thm1}
Assuming $f_2(x) = \frac{1}{n}\sum_{i = 1}^{n} \phi_i(x) + \frac{\nu}{2} \lVert x \rVert_2^2$ ~for some $\nu > 0$ and $\nabla \phi_i(x),i = 1,\cdots,n$ are uniformly bounded. If we choose $c > 0,\alpha \in (0.5 ,1]$ and $0< \lambda < \frac{1}{\rho(BB^T)}$, then
\begin{equation}\label{thm1eq1}
\underset{k \to \infty}{\underline{\lim}}\mathbb{E}^{(k)}\big(\lVert x_k - x^* \rVert_2^2 \big) \to 0.
\end{equation}
\end{theorem}
\begin{proof}
 Denote $l(x) = \frac{1}{n}\sum_{i = 1}^{n} \phi_i(x)$ and $k_0$ to be the smallest number such that $\gamma_{k_0} < \frac{1}{2\nu}$; then, by Eq. \textbf{(\ref{cglmeq11})}, we have, for $k > k_0$,
\begin{equation}\label{thm1eq2}
\begin{aligned}
a_{k + 1}
&\overset{\tiny{\circled{1}}}{\leq} \mathbb{E}^{(k)}\big(\lVert x_k - x^* \rVert_2^2\big) + \frac{\gamma_{k}^2}{\lambda}(1 - \lambda \rho_{\min}(B B^T)) \mathbb{E}^{(k)}( \lVert v_k - v^* \rVert_2^2 ) \\
& - 2\gamma_k \mathbb{E}^{(k)}\big< \nabla f_2(x_k) - \nabla f_2(x^*), x_k - x^*\big> + \gamma_k^2 \mathbb{E}^{(k + 1)}
\Big(\lVert \nabla f_2^{[i_k]}(x_k) - \nabla f_2(x^*) \rVert_2^2\Big)  \\
& \leq a_k - 2\gamma_k \mathbb{E}^{(k)}\big< \nabla f_2(x_k) - \nabla f_2(x^*), x_k - x^*\big> \\
& \qquad\qquad\qquad  + \gamma_k^2 \mathbb{E}^{(k + 1)}\Big(\lVert \nabla l^{[i_k]}(x_k) + \nu x_k - \nabla l(x^*) - \nu x^*\rVert_2^2\Big)  \\
& \overset{\tiny{\circled{2}}}{\leq} a_k - 2\gamma_k \mathbb{E}^{(k)}\big< \nabla f_2(x_k) - \nabla f_2(x^*), x_k - x^*\big> + 2\nu^2\gamma_k^2\mathbb{E}^{(k)}(\lVert x_k - x^*\rVert_2^2)\\
& \qquad \qquad \qquad + 2\gamma_k^2\mathbb{E}^{(k + 1)}(\lVert \nabla l^{[i_k]}(x_k)- \nabla l(x^*)\rVert_2^2)\\,
& \overset{\tiny{\circled{3}}}{\leq} a_k - 2\gamma_k \mathbb{E}^{(k)}\big< \nabla f_2(x_k) - \nabla f_2(x^*), x_k - x^*\big> + 2\nu^2\gamma_k^2\mathbb{E}^{(k)}(\lVert x_k - x^*\rVert_2^2)\\
& \qquad \qquad \qquad  + 4\gamma_k^2\big(\mathbb{E}^{(k + 1)}(\lVert \nabla l^{[i_k]}(x_k)\rVert_2^2) + \lVert \nabla l(x^*)\rVert_2^2\big)\\
& \overset{\tiny{\circled{4}}}{\leq} a_k - 2\gamma_k \mathbb{E}^{(k)}\big< \nabla f_2(x_k) - \nabla f_2(x^*), x_k - x^*\big> + 2\nu^2\gamma_k^2\mathbb{E}^{(k)}(\lVert x_k - x^*\rVert_2^2) + \gamma_k^2 M\\
& \overset{\tiny{\circled{5}}}{\leq} a_k - 2\nu\gamma_k(1 - \nu \gamma_k)\mathbb{E}^{(k)}(\lVert x_k - x^*\rVert_2^2) + \gamma_k^2 M\\
& \overset{\tiny{\circled{6}}}{\leq} a_k - \nu\gamma_k\mathbb{E}^{(k)}(\lVert x_k - x^*\rVert_2^2) + \gamma_k^2 M\\,
\end{aligned}
\end{equation}
where
\begin{itemize}
  \item Inequality \small{\circled{1}} follows from Lemma \ref{cglm1}.
  \item Inequalities \small{\circled{2}} and \small{\circled{3}} use the inequality $\lVert a + b \rVert_2^2 \leq 2\lVert a \rVert_2^2 + 2\lVert b \rVert_2^2$.
  \item Inequality \small{\circled{4}} uses the fact that the $\nabla l(x)$ is uniformly bounded, i.e., there must be a constant $M$ such that $4\big(\mathbb{E}^{(k + 1)}\lVert \nabla l^{[i_k]}(x_k)\rVert_2^2) + \lVert \nabla l(x^*)\rVert_2^2\big) \leq M$.
  \item Inequality \small{\circled{5}} uses the strong convexity of function $f_2(x)$.
  \item Inequality \small{\circled{6}} uses the fact that $2\nu\gamma_k(1 - \nu \gamma_k) \geq \nu\gamma_k$ since $\gamma_k \leq \gamma_{k_0} \leq \frac{1}{2\nu}$ for $k \geq k_0$.
\end{itemize}
Summing Eq. \textbf{(\ref{thm1eq2})} from $k_0$ to $k$, we have
\begin{equation}\label{thm1eq33}
\begin{aligned}
\sum_{k = k_0}^{k} \nu \gamma_k\mathbb{E}^{(k)}(\lVert x_k - x^*\rVert_2^2)
& \leq a_{k_0} + M\sum_{k = k_0}^{k}\gamma_k^2 - a_{k + 1} \\
& \leq a_{k_0} + M\sum_{k = k_0}^{k}\gamma_k^2.
\end{aligned}
\end{equation}
Letting $k \rightarrow +\infty$, we obtain
\begin{equation}\label{thm1eq3}
\sum_{k = k_0}^{+ \infty} \nu \gamma_k\mathbb{E}^{(k)}(\lVert x_k - x^*\rVert_2^2)
\leq a_{k_0} + M\sum_{k = k_0}^{+ \infty}\gamma_k^2.
\end{equation}

Since $\alpha \in (0.5,1]$, we have $\sum_{k = k_0}^{\infty} \gamma_k = \infty$ and $\sum_{k = k_0}^{\infty} \gamma_k^2 < \infty$. Thus, if $\underset{k \to \infty}{\underline{\lim}}\mathbb{E}^{(k)}(\lVert x_k - x^*\rVert_2^2 \not \to 0$, then there must be a constant $c > 0$ such that $\mathbb{E}^{(k)}(\lVert x_{k} - x^*\rVert_2^2) > c, \forall k \geq k_0$; we then have
\begin{equation}\label{thm1eq4}
\infty = c\nu\sum_{k = k_0}^{\infty}\gamma_{k}
< \sum_{k = k_0}^{\infty}\gamma_{k}\nu\mathbb{E}^{(k)}(\lVert x_{k} - x^*\rVert_2^2) < a_{k_0} + M\sum_{k = k_0}^{\infty}\gamma_{k}^2 < \infty.
\end{equation}
Contradiction. Thus, $\underset{k \to \infty}{\underline{\lim}}\mathbb{E}^{(k)}(\lVert x_k - x^*\rVert_2^2 \to 0$.
\qquad
\end{proof}

\begin{remark}
It can be seen that Theorem \textbf{\ref{thm1}} holds if we add a $l_2$ term on logistic loss or hinge loss.
\end{remark}
Theorem \textbf{\ref{thm1}} gives the convergence of SPDFP, but does not provide the convergence rate. The following theorem gives the convergence rate of SPDFP.

\begin{theorem}\label{thm2}
We assume that $f_2(x) = \frac{1}{n}\sum_{i = 1}^{n} \phi_i(x) + \frac{\nu}{2} \lVert x \rVert_2^2$ for some $\nu > 0$ and $\nabla \phi_i(x)$ is uniformly bounded (related to $M$). Furthermore, the matrix $B$ has full row rank. Given $c > 0, \alpha \in (0,1],0 < \lambda < \frac{1}{\rho_{max}(BB^T)}$ in Algorithm 1, and letting $k_0 > 0$ be an integer large enough such that $\gamma_{k_0} \leq \min\{ \frac{1}{2\nu} and \frac{\lambda \rho_{min}(BB^T)}{\nu}\}$, then the following estimate holds:
\begin{equation}\label{thm2eqness}
\begin{aligned}
a_{k + 1} \leq (1 - \nu\gamma_k)a_k + \gamma_k^2M, \qquad \qquad \forall k \geq k_0.
\end{aligned}
\end{equation}
\end{theorem}
\begin{proof}
From the proof of \textbf{(\ref{thm1eq2})}, we have
for $k > k_0$
\begin{equation}\label{thm2eq2}
\begin{aligned}
a_{k + 1} &\leq \mathbb{E}^{(k)}\big(\lVert x_k - x^* \rVert_2^2\big) - \nu\gamma_k\mathbb{E}^{(k)}(\lVert x_k - x^*\rVert_2^2) + \gamma_k^2 M
+ \frac{\gamma_{k}^2}{\lambda}(1 - \lambda \rho_{\min}(B B^T)) \mathbb{E}^{(k)}( \lVert v_k - v^* \rVert_2^2 )  \\
& = (1 - \nu\gamma_k)\mathbb{E}^{(k)}(\lVert x_k - x^*\rVert_2^2)
+ \frac{\gamma_{k}^2}{\lambda}(1 - \lambda \rho_{\min}(B B^T)) \mathbb{E}^{(k)}( \lVert v_k - v^* \rVert_2^2 ) + \gamma_k^2 M \\
& \leq (1 - \nu \gamma_k)\mathbb{E}^{(k)}\Big(\lVert x_k - x^* \rVert_2^2 + \frac{\gamma_k^2}{\lambda} \lVert v_k - v^* \rVert_2^2 \Big) + \gamma_k^2M \\
& = (1 - \nu \gamma_k)a_k + \gamma_k^2M \\.
\end{aligned}
\end{equation}
The first inequality of \textbf{(\ref{thm2eq2})} holds since we replace $a_k$ by $\mathbb{E}^{(k)}\big(\lVert x_k - x^* \rVert_2^2\big) + \frac{\gamma_k^2}{\lambda}(1 - \lambda \rho_{\min}(B B^T))$ $	\mathbb{E}^{(k)}( \lVert v_k - v^* \rVert_2^2 )$ in Eq. \textbf{(\ref{thm1eq2})} and use the fact that $\gamma_{k_0} \leq \frac{1}{2\nu}$. The second inequality follows from the fact that $\gamma_{k} \leq \gamma_{k_0} \leq \frac{\lambda \rho_{min}(BB^T)}{\nu}, \forall k \geq k_0$.
\end{proof}

\begin{corollary}\label{thm1cor}
If we let $\eta_k = \nu \gamma_k$ and $D = \frac{M}{\nu^2}$ and $c_0 = \nu c$, then Eq. \textbf{(\ref{thm2eqness})} becomes
\begin{equation}\label{thm2eqness1}
\begin{aligned}
a_{k + 1} \leq (1 - \eta_k)a_k + \eta_k^2D, \qquad \qquad \forall k \geq k_0.
\end{aligned}
\end{equation}
Using Lemma \textbf{\ref{lm3}} and the fact that $\eta_{k_0} = \frac{\nu c }{k_0^{\alpha}} \leq \frac{1}{2} < 1$ to Eq. \textbf{(\ref{thm2eqness1})}, we obtain
\begin{equation}\label{thm2eqness2}
a_{k + 1} \leq \left\{
\begin{aligned}
& \Big( D c_0^2 \varphi_{1 - 2\alpha}(k) + a_{k_0}\exp \Big( \frac{ c_0 k_0^{1 - \alpha}}{1 - \alpha} \Big)\Big) \exp\Big( \frac{-c_0 (1 - 2^{\alpha - 1})(k + 1)^{1 - \alpha}}{1 - \alpha} \Big) + \frac{D2^{\alpha} c_0 }{(k - 2)^{\alpha}} \\
&\qquad\qquad\qquad\qquad\qquad\qquad\qquad\qquad\qquad\qquad if ~ \alpha \in (0,1). \\
& a_{k_0} \Big(\frac{k_0}{k + 1}\Big)^{c_0}  + \frac{D c_0^2}{(k + 1)^{c_0}}(1 + \frac{1}{k_0})^{c_0} \varphi_{c_0 - 1}(k) \qquad if ~ \alpha = 1
\end{aligned}
\right.
\end{equation}
for $k \geq 2k_0$. This gives the $\mathcal{O}{(k^{-\alpha}})$ convergence rate of SPDFP.\\
\end{corollary}

\subsection{Non-boundness of gradient.}
In the preceding section, we gave the convergence analysis based on the bounded gradient of the first term of $f_2$; however, if it is not true but the following assumption holds, then there are also convergence results of SPDFP.
\begin{assumption}\label{assmpthm3}
There exist $C_1,C_2 > 0$ such that the following inequality holds:
\begin{equation}\label{assump3eq}
\mathbb{E}^{(k + 1)}\big(\lVert \nabla f_2^{[i_k]}(x_k) - \nabla f_2(x_k) \rVert_2^2\big) \leq C_1\mathbb{E}^{(k)}(\lVert x_k \rVert_2^2) + C_2,
\end{equation}
where $x_k$ denotes the $k$th iteration of Algorithm 1.
\end{assumption}

\noindent It can be seen that the left-hand side of Eq. \textbf{(\ref{assump3eq})} is
\begin{equation*}
\begin{aligned}
\mathbb{E}^{(k + 1)}\big(\lVert \nabla f_2^{[i_k]}(x_k) - \nabla f_2(x_k) \rVert_2^2\big)
& = \mathbb{E}^{(k + 1)}\big(\lVert \nabla f_2^{[i_k]}(x_k) \rVert_2^2 \big) - \mathbb{E}^{(k)}\big(\lVert \nabla f_2(x_k) \rVert_2^2\big) \\
& \leq \mathbb{E}^{(k + 1)}\big(\lVert \nabla f_2^{[i_k]}(x_k) \rVert_2^2 \big).
\end{aligned}
\end{equation*}
Thus, to satisfy inequality \textbf{(\ref{assump3eq})}, we only need to verify
\begin{equation}\label{ammp3rkeq}
\mathbb{E}^{(k + 1)}\big(\lVert \nabla f_2^{[i_k]}(x_k) \rVert_2^2 \big)
 \leq C_1\mathbb{E}^{(k)}(\lVert x_k \rVert_2^2) + C_2.
\end{equation}
If $f_2(x)$ is square loss i.e., $f_2(x) = \frac{1}{n}\lVert Ax - b \rVert_2^2 = \frac{1}{n}\sum_{i = 1}^{n}(a_i^Tx - b(i))^2$ [here, $a_i^T$ is the $i$th row of matrix $A$ and $b(i)$ is the $i$th component of vector $b$], then by the definition of $\nabla f_2^{[i_k]}(x_k)$ we have
\begin{equation}
\begin{aligned}
\mathbb{E}^{(k + 1)}\big(\lVert \nabla f_2^{[i_k]}(x_k) \rVert_2^2 \big)
& = \sum_{j = 1}^{n/p}\frac{p}{n}\frac{1}{p^2}\mathbb{E}^{(k)}\big(\lVert A_j^T(A_jx_k - b_j) \rVert_2^2\big) \\
& \leq \frac{1}{np}\sum_{j = 1}^{n/p}\rho_{\max}(A_j A_j^T)\mathbb{E}^{(k)}\big(\lVert(A_jx_k - b_j) \rVert_2^2\big) \\
& \leq  \frac{L_p}{np}\sum_{j = 1}^{n/p}\mathbb{E}^{(k)}\big(\lVert A_jx_k\rVert_2 + \lVert b_j \rVert_2 \big)^2 \\
& \leq  \frac{L_p}{np}\sum_{j = 1}^{n/p} 2\Big( \rho_{\max}(A_j^T A_j)\mathbb{E}^{(k)}\big(\lVert x_k\rVert_2^2\big) + \lVert b_j \rVert_2^2\Big) \\
& \leq \frac{2L_p^2}{p^2}\mathbb{E}^{(k)}\big(\lVert(x_k\rVert_2^2\big) + \frac{2L_p}{np}\lVert b \rVert_2^2,
\end{aligned}
\end{equation}
where
\begin{itemize}
  \item $A_j$ is the sub-matrix of $A$ drawn from $(j - 1)*p + 1$ to $j*p + 1$ rows and $b_j$ is a subvector of $b$ drawn from same index from $b$.
  \item $L_p = \max \{ \rho_{\max}(A_1 A_1^T), \cdots, \rho_{\max}(A_{n/p} A_{n/p}^T) \}$.
\end{itemize}
Thus, we can choose $C_1 = \frac{2L_p^2}{p^2}$
and $C_2 = \frac{2L_p}{np}\lVert b \rVert_2^2$ in Assumption \textbf{\ref{assmpthm3}}.
We now give the convergence and convergence rate of SPDFP based on Assumption \textbf{\ref{assmpthm3}}.
\begin{theorem}\label{thm3}
Assuming $f_2(x)$ is $\nu$-strongly convex, $\frac{1}{\beta}$ is the continuous gradient and Assumption \textbf{\ref{assmpthm3}} holds. If we choose $\alpha \in (0.5 ,1],\gamma_k = \frac{c}{k^{\alpha}}$ and $0 < \lambda \leq \frac{1}{\rho(BB^T)}$, then
\begin{equation}\label{thm2eq1}
\underset{k \to \infty}{\underline{\lim}}\mathbb{E}^{(k)}\Big(\lVert x_k - x^* \rVert_2^2 \Big) \to 0.
\end{equation}
\end{theorem}
\begin{proof}
As before, letting $k_0$ be large enough such that $\gamma_{k_0} < \frac{\beta\nu}{2(\nu + 2C_1\beta)}$, then by Eq. \textbf{(\ref{cglmeq11})} we have, for $k > k_0$,
\begin{equation}\label{thm3eq2}
\begin{aligned}
a_{k + 1}
& \overset{\tiny{\circled{1}}}{\leq} a_k - 2\gamma_k \mathbb{E}^{(k)}\big< \nabla f_2(x_k) - \nabla f_2(x^*), x_k - x^*\big> + \gamma_k^2 \mathbb{E}^{(k + 1)}(\lVert \nabla f_2^{[i_k]}(x_k) - \nabla f_2(x^*) \rVert_2^2) \\
& = a_k - 2\gamma_k \mathbb{E}^{(k)}\big< \nabla f_2(x_k) - \nabla f_2(x^*), x_k - x^*\big> \\
& \qquad \qquad \qquad + \gamma_k^2 \mathbb{E}^{(k + 1)}(\lVert \nabla f_2^{[i_k]}(x_k) - \nabla f_2(x_k) + \nabla f_2(x_k) - \nabla f_2(x^*) \rVert_2^2)\\
& \overset{\tiny{\circled{2}}}{\leq} a_k - 2\gamma_k \mathbb{E}^{(k)}\big< \nabla f_2(x_k) - \nabla f_2(x^*), x_k - x^*\big> + 2\gamma_k^2 \mathbb{E}^{(k + 1)}(\lVert \nabla f_2^{[i_k]}(x_k) - \nabla f_2(x_k) \rVert_2^2) \\
& \qquad \qquad+ 2\gamma_k^2 \mathbb{E}^{(k)}(\lVert \nabla f_2(x_k) - \nabla f_2(x^*) \rVert_2^2) \\
& \overset{\tiny{\circled{3}}}{\leq} a_k - 2\gamma_k(1 - \frac{\gamma_k}{\beta}) \mathbb{E}^{(k)}\big< \nabla f_2(x_k) - \nabla f_2(x^*), x_k - x^*\big> + 2\gamma_k^2\big(C_1\mathbb{E}^{(k)}(\lVert x_k \rVert_2^2) + C_2\big) \\
& \overset{\tiny{\circled{4}}}{\leq} a_k - 2\gamma_k(1 - \frac{\gamma_k}{\beta})\nu  \mathbb{E}^{(k)}(\lVert x_k - x^* \rVert_2^2) + 2\gamma_k^2C_1\mathbb{E}^{(k)}\big(\lVert x_k - x^* + x^* \rVert_2^2\big) + 2\gamma_k^2C_2,
\end{aligned}
\end{equation}
\begin{equation*}
\begin{aligned}
& \overset{\tiny{\circled{5}}}{\leq} a_k - 2\gamma_k\big(1 - (\frac{1}{\beta} + \frac{2C_1}{\nu})\gamma_k\big)\nu  \mathbb{E}^{(k)}(\lVert x_k - x^* \rVert_2^2) + \gamma_k^2\big(4C_1\lVert x^* \rVert_2^2 + 2C_2\big)\\
& = a_k - 2\gamma_k\big(1 - \frac{\nu + 2C_1\beta}{\beta\nu}\gamma_k\big)\nu  \mathbb{E}^{(k)}(\lVert x_k - x^* \rVert_2^2) + \gamma_k^2\big(4C_1\lVert x^* \rVert_2^2 + 2C_2\big)\\
& \overset{\tiny{\circled{6}}}{\leq} a_k - \gamma_k\nu  \mathbb{E}^{(k)}(\lVert x_k - x^* \rVert_2^2) + \gamma_k^2\big(4C_1\lVert x^* \rVert_2^2 + 2C_2\big),
\end{aligned}
\end{equation*}
where
\begin{itemize}
  \item Inequality \small{\circled{1}} follows from Lemma \ref{cglm1}.
  \item Inequality \small{\circled{2}} and \small{\circled{5}} use the inequality $\lVert a + b\rVert_2^2 \leq 2\lVert a \rVert_2^2 + 2\lVert b \rVert_2^2$.
  \item Inequality \small{\circled{3}} uses Assumption \ref{assmpthm3} and $\frac{1}{\beta}$ is the continuous gradient of $f_2(x)$.
  \item Inequality \small{\circled{4}} uses the $\nu$-strong convexity of function $f_2(x)$ and the fact that $$\gamma_k \leq \gamma_{k_0} \leq \frac{\beta\nu}{2(\nu + 2C_1\beta)} \leq \beta.$$
  \item Inequality \small{\circled{6}} uses the fact that $2\gamma_k\big(1 - (\frac{\nu + 2C_1\beta}{\beta\nu})\gamma_k\big)\nu \geq \nu \gamma_k $ since $\gamma_k \leq \gamma_{k_0} \leq \frac{\beta\nu}{2(\nu + 2C_1\beta)},\forall k \geq k_0$.
\end{itemize}
Similar to Theorem \textbf{\ref{thm1}}, we obtain the result by contradiction. \qquad
\end{proof}
For the convergence rate, we have the following theorem.
\begin{theorem}\label{thm4}
Assuming $f_2(x)$ is $\nu$-strongly convex $\frac{1}{\beta}$ smooth, Assumption \textbf{\ref{assmpthm3}} holds. Furthermore, the matrix $B$ has full row rank. Given $c > 0, \alpha \in (0,1],0 < \lambda < \frac{1}{\rho_{max}(BB^T)}$, and letting $k_0 > 0 $ be an integer large enough such that $\gamma_{k_0} \leq \min\{ \frac{\beta\nu}{2(\nu + 2C_1\beta)}, \frac{\lambda \rho_{min}(BB^T)}{\nu}\}$, then the following estimate holds:
\begin{equation}\label{thm4eqness}
\begin{aligned}
a_{k + 1} \leq (1 - \nu\gamma_k)a_k + \gamma_k^2M_1, \qquad \qquad \forall k \geq k_0,
\end{aligned}
\end{equation}
where $M_1 = 4C_1\lVert x^* \rVert_2^2 + 2C_2$.
\end{theorem}

\begin{proof}
Recalling Eq. \textbf{(\ref{thm3eq2})}, we have, for $k > k_0$,
\begin{equation}\label{thm4eq2}
\begin{aligned}
a_{k + 1} &\leq \mathbb{E}^{(k)}\big(\lVert x_k - x^* \rVert_2^2\big)
+ \frac{\gamma_{k}^2}{\lambda}(1 - \lambda \rho_{min}(BB^T)) \mathbb{E}^{(k)}( \lVert v_k - v^* \rVert_2^2 ) - \nu\gamma_k\mathbb{E}^{(k)}(\lVert x_k - x^*\rVert_2^2)\\
& \qquad + \gamma_k^2\big(4C_1\lVert x^* \rVert_2^2 + 2C_2\big)  \\
& = (1 - \nu\gamma_k)\mathbb{E}^{(k)}(\lVert x_k - x^*\rVert_2^2)
+ \frac{\gamma_{k}^2}{\lambda}(1 - \lambda \rho_{min}(BB^T)) \mathbb{E}^{(k)}( \lVert v_k - v^* \rVert_2^2 ) + \gamma_k^2\big(4C_1\lVert x^* \rVert_2^2 + 2C_2\big)  \\
& \leq (1 - \nu \gamma_k)\mathbb{E}^{(k)}\Big(\lVert x_k - x^* \rVert_2^2 + \frac{\gamma_k^2}{\lambda} \lVert v_k - v^* \rVert_2^2 \Big) + \gamma_k^2\big(4C_1\lVert x^* \rVert_2^2 + 2C_2\big) \\
& = (1 - \nu \gamma_k)a_k + \gamma_k^2\big(4C_1\lVert x^* \rVert_2^2 + 2C_2\big) \\
& = (1 - \nu \gamma_k)a_k + \gamma_k^2M_1.
\end{aligned}
\end{equation}
In the first inequality, we replace $a_k$ by $\mathbb{E}^{(k)}\big(\lVert x_k - x^* \rVert_2^2\big) + \frac{\gamma_k^2}{\lambda}(1 - \lambda \rho_{min}(B B^T)) \mathbb{E}^{(k)}( \lVert v_k - v^* \rVert_2^2 )$ in Eq. \textbf{(\ref{thm3eq2})} and use the fact that $\gamma_{k_0} \leq \frac{\beta\nu}{2(\nu + 2C_1\beta)}$. The second inequality follows from the fact that $\gamma_k \leq \gamma_{k_0} \leq \frac{\lambda \rho_{min}(BB^T)}{\nu}, \forall k \geq k_0$.
\end{proof}

\begin{corollary}\label{thm4cor}
If we let $\eta_k = \nu \gamma_k$ and $D = \frac{4C_1\lVert x^* \rVert_2^2 + 2C_2}{\nu^2},c_0 = \nu c$, then Eq. \textbf{(\ref{thm2eqness})} become
\begin{equation}\label{thm4eqness1}
\begin{aligned}
a_{k + 1} \leq (1 - \eta_k)a_k + \eta_k^2D, \qquad \qquad \forall k \geq k_0.
\end{aligned}
\end{equation}
Letting $K = \max \{ k_0, k_1\}$, where $k_1$ is the smallest integer such that $\eta_{k_1} < 1$, then, by using Lemma \textbf{\ref{lm3}} to Eq. \textbf{(\ref{thm2eqness1})}, we obtain
\begin{equation}\label{thm4eqness2}
a_{k + 1} \leq \left\{
\begin{aligned}
& \Big( D c_0^2 \varphi_{1 - 2\alpha}(k) + a_{K}exp \Big( \frac{c_0K^{1 - \alpha}}{1 - \alpha} \Big)\Big) exp\Big( \frac{-c_0(1 - 2^{\alpha - 1})(k + 1)^{1 - \alpha}}{1 - \alpha} \Big) + \frac{D2^{\alpha}c_0}{(k - 2)^{\alpha}} \\
&\qquad\qquad\qquad\qquad\qquad\qquad\qquad\qquad\qquad\qquad if ~ \alpha \in (0,1). \\
& a_{K} \Big(\frac{K}{k + 1}\Big)^{c_0}  + \frac{D c_0^2}{(k + 1)^{c_0}}(1 + \frac{1}{K})^{c_0} \varphi_{c_0 - 1}(k) \qquad if ~ \alpha = 1
\end{aligned}
\right.
\end{equation}
for $k \geq 2K$.
\end{corollary}

Table \textbf{\ref{cgsum}} summarizes the convergence results of SPDFP based on the following:
\begin{itemize}
	\item Strong convexity of $f_2(x)$ (S.C.)
	\item Bounded gradient (BG)
	\item Assumption \textbf{\ref{assmpthm3}} (Asmp \textbf{\ref{assmpthm3})}
	\item Full row rank of matrix $B$ (FrkB)
	\item Lipchitz continous gradient of $f_2(x)$ (Lip)
	\item Range of $\alpha$ ($\alpha$)
	\item Convergence (Cg)
	\item Convergence rate (Cg rate)
\end{itemize}

\begin{table}[htp]
\caption{Summary of convergence results of SPDFP}
\centering
\begin{tabular}{|c|c|c|c|c|c|c|c|c|c|c|c|c|c|c|c|c|}
\hline
\multicolumn{1}{|c|}{Case}
& \multicolumn{2}{|c|}{S.C.}
& \multicolumn{2}{|c|}{BG}
& \multicolumn{2}{|c|}{Asmp \textbf{\ref{assmpthm3}}}
& \multicolumn{2}{|c|}{Lip}
& \multicolumn{2}{|c|}{FrkB}
& \multicolumn{2}{|c|}{$\alpha$}
& \multicolumn{2}{|c|}{Cg}
& \multicolumn{2}{|c|}{Cg rate} \\
\hline
\multicolumn{1}{|c|}{\textbf{Theorem} \textbf{\ref{thm1}}}
& \multicolumn{2}{|c|}{$\checkmark$}
& \multicolumn{2}{|c|}{$\checkmark$}
& \multicolumn{2}{|c|}{$-$}
& \multicolumn{2}{|c|}{$-$}
& \multicolumn{2}{|c|}{$-$}
& \multicolumn{2}{|c|}{$(0.5,1]$}
& \multicolumn{2}{|c|}{$\checkmark$}
& \multicolumn{2}{|c|}{$-$} \\
\hline
\multicolumn{1}{|c|}{\textbf{Theorem} \textbf{\ref{thm2}}}
& \multicolumn{2}{|c|}{$\checkmark$}
& \multicolumn{2}{|c|}{$\checkmark$}
& \multicolumn{2}{|c|}{$-$}
& \multicolumn{2}{|c|}{$-$}
& \multicolumn{2}{|c|}{$\checkmark$}
& \multicolumn{2}{|c|}{$(0,1]$}
& \multicolumn{2}{|c|}{$\checkmark$}
& \multicolumn{2}{|c|}{$\checkmark$} \\
\hline
\multicolumn{1}{|c|}{\textbf{Theorem} \textbf{\ref{thm3}}}
& \multicolumn{2}{|c|}{$\checkmark$}
& \multicolumn{2}{|c|}{$-$}
& \multicolumn{2}{|c|}{$\checkmark$}
& \multicolumn{2}{|c|}{$\checkmark$}
& \multicolumn{2}{|c|}{$-$}
& \multicolumn{2}{|c|}{$(0.5,1]$}
& \multicolumn{2}{|c|}{$\checkmark$}
& \multicolumn{2}{|c|}{$-$} \\
\hline
\multicolumn{1}{|c|}{\textbf{Theorem} \textbf{\ref{thm4}}}
& \multicolumn{2}{|c|}{$\checkmark$}
& \multicolumn{2}{|c|}{$-$}
& \multicolumn{2}{|c|}{$\checkmark$}
& \multicolumn{2}{|c|}{$\checkmark$}
& \multicolumn{2}{|c|}{$\checkmark$}
& \multicolumn{2}{|c|}{$(0,1]$}
& \multicolumn{2}{|c|}{$\checkmark$}
& \multicolumn{2}{|c|}{$\checkmark$} \\
\hline
\end{tabular}
\label{cgsum}
\end{table}
\begin{remark}\label{Dis1}
It can be seen from Eq. \textbf{(\ref{thm2eqness2})} and Remark \textbf{\ref{rk1}} that the best convergence rate of SPDFP is $\mathcal{O}({1/k})$ (when $\alpha = 1,c > 1$).
However, the number $k_0$ in Theorem \textbf{\ref{thm2}} - Theorem \textbf{\ref{thm4}} may be too large if the condition number of matrix $BB^T$ and $f_2(x)$ is bad. When this happens, a larger $\alpha$, which means faster decreasing step size, will diminish the overall performance of the algorithm. Thus, in our second and third numerical examples, we choose $\alpha < 1$, especially slightly larger than $0.5$ to solve some real-world datasets (see Section 5).
\end{remark}
\section{Numerical results.}
\indent In this section, we investigate the numerical performance of SPDFP. First, we synthesized an example called fussed lasso to see the behavior of SPDFP with different step sizes (i.e., different $\alpha$) to confirm the correctness of the convergence results. We then performed experiments on graph guide SVM \cite{STOCADMM}. The comparison between SADMM \cite{STOCADMM} and SPDFP on dataset 20newsgroups\footnote{http://www.cs.nyu.edu/~roweis/data.html} for a multi-class classification task will be given. Finally, we performed experiments on graph guide logistic regression \cite{Kim}. The comparisons between OPG-ADMM \cite{RDAOPGADMM}, SCAS-ADMM \cite{SCASADMM}, PDFP \cite{PDFP}, and SPDFP on datasets \textbf{a9a} and \textbf{covtype} from \textbf{LIBSVM} \cite{CC01a} will be presented.
\subsection{Synthetic example.}
\noindent First, we consider a synthetized example called fussed lasso, i.e.,
\begin{equation}\label{fussed}
\underset{x \in \mathbb{R}^d}{\min}~ \frac{1}{2n}\lVert Ax - b \rVert_2^2 + \nu \lVert Bx \rVert_1,
\end{equation}
where $A \in \mathbb{R}^{n \times d}, x \in \mathbb{R}^d,b \in \mathbb{R}^n$,
and matrix $B \in \mathbb{R}^{(n - 1) \times n}$ is a sparse matrix, the diagonal entry of which is $-1$; the upper diagonal entry is $1$ and all the other entries are zero. The first term is a data-fidelity term and the second term ensures the sparsity of successive differences in $x$.

We synthesize the problem as follows. The entries of $A$ are drawn from standard normal distribution. Its dimension is $n = 10000$ and $d = 200$. For vector $b$, we first draw a vector $x_0$, the entries of which are $1$. We then randomly choose $0.05$ of its entries to be perturbed by noise; after that, the vector $b$ is computed by $b = Ax_0 + \varepsilon$, where $\varepsilon$ is random Gaussian noise.
The ground truth of ({\ref{fussed}}) is then obtained by running PDFP for 3000 iterations for which we observed the convergence.

\begin{figure}[!htbp]
\centering
\includegraphics[width=0.45\columnwidth]{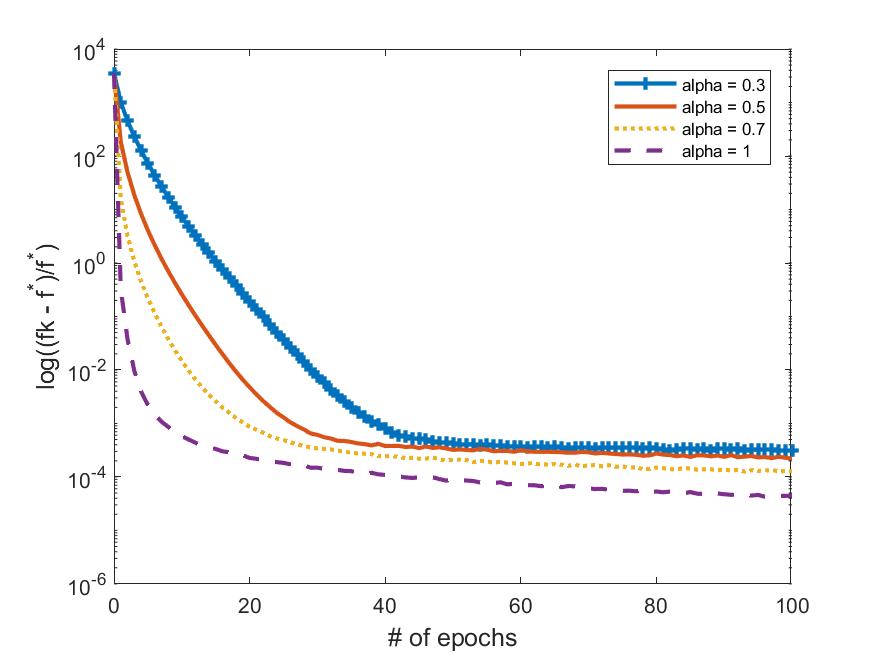}%
\includegraphics[width=0.45\columnwidth]{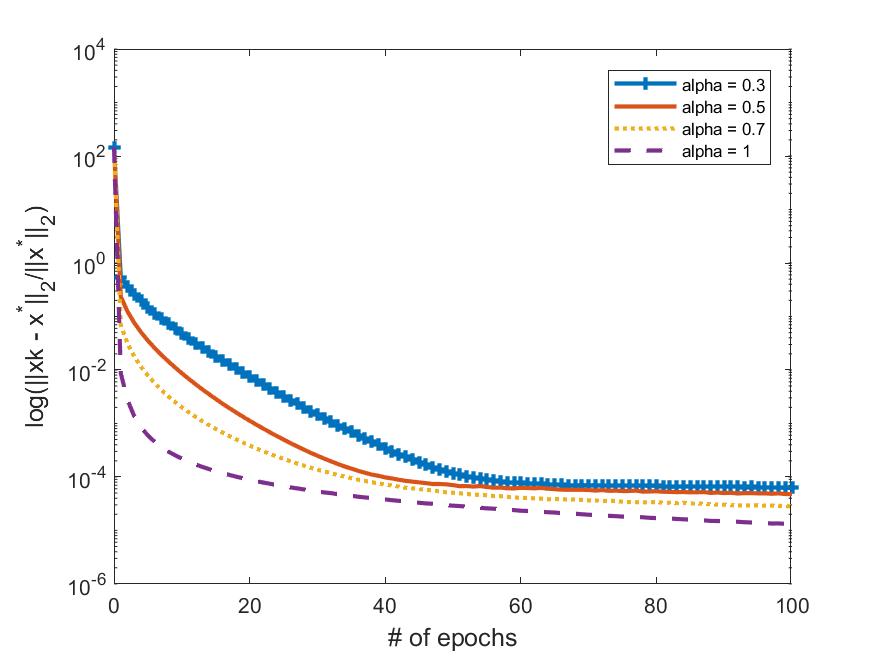}
\caption{Relative error of the function value (left) and of the iterate (right) \textcolor{red}{vs} epoch over 10 independent repetitions  }
\label{fl1}
\end{figure}

\begin{figure}[!htbp]
\centering
\includegraphics[width=0.45\columnwidth]{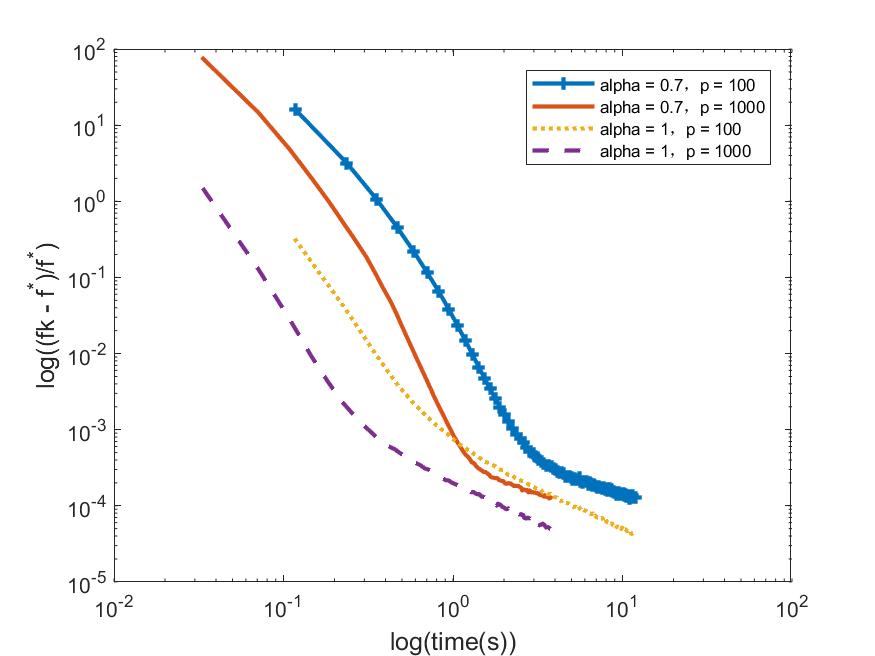}%
\includegraphics[width=0.45\columnwidth]{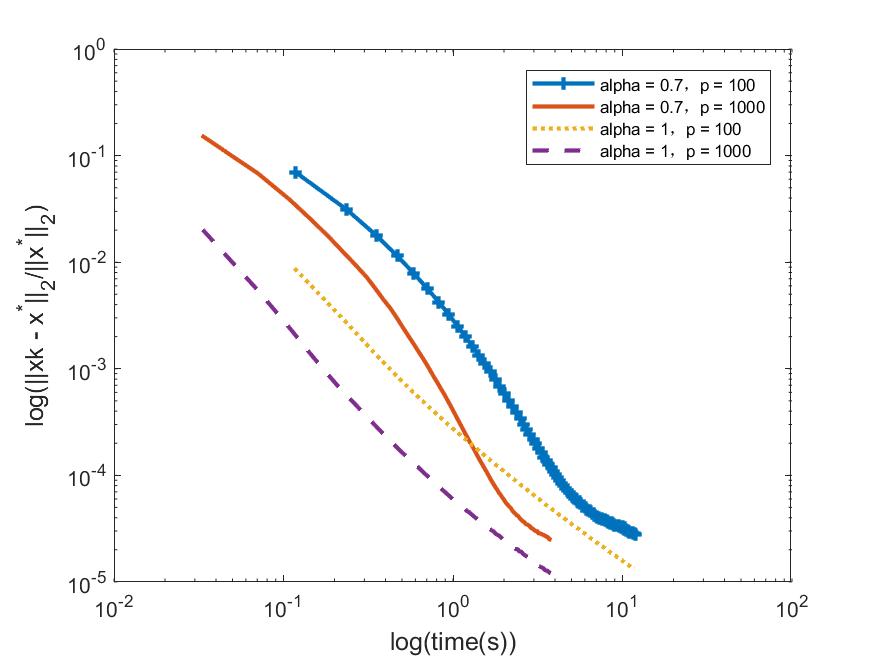}
\caption{Relative error of the function value (left) and of the iterate (right) \textcolor{red}{vs} time with different batch sizes and $\alpha$ over 10 independent repetitions}
\label{fl2}
\end{figure}

Figure \textbf{\ref{fl1}} gives the relative error of function value and error of iteration. In each figure, we consider different step sizes, i.e., $\gamma_k = c/k^{0.3}$,$c/k^{0.5}$, $c/k^{0.7}$, $c/k$ (i.e., $\alpha = 0.3,0.5,0.7,1$). The constant $c$ is given such that it give the best performance for a given $\alpha$. The batch size is $100$ here.  From the results we can see that example with bigger $\alpha$ have better performance both for relative error of the function value and of the iterate, which conforms to our convergence analysis. Figure \textbf{\ref{fl2}} gives the relative error of the function value and of the iterate versus time. We consider different step sizes ($\alpha = 0.7, 1$) and batch sizes ($p = 100,1000$). It can be seen that the examples with larger $\alpha$, i.e., $\alpha = 1$, and batch size, i.e., $p = 1,000$, exhibit the best performance.\\

\subsection{Graph-Guide SVM.}
The second example we considered is a problem called graph guide SVM, which was also considered in \cite{STOCADMM}. The problem is formulated as follows:
\begin{equation}\label{GGSVM1}
\underset{x \in \mathbb{R}^d}{\min}~\frac{1}{n}\sum_{i = 1}^{n}\phi_i(x) + \gamma\lVert x \rVert_2^2 + \nu\lVert Bx \rVert_1,
\end{equation}
where $\phi_i(x) = \max\{ 0,1 - b_is_i^Tx \}, i = 1,\cdots, n$ the hinge loss, and $b_i$ denotes the label of the $i$th sample $s_i$. The matrix $B$ is determined by sparse inverse covariance selection \cite{Banerjee2008model} (or graphical lasso \cite{Friedman2008sparse}).
In \cite{STOCADMM}, the author uses stochastic ADMM to solve this problem, i.e., the problem \textbf{(\ref{GGSVM1})} is reformulated as
\begin{equation}\label{GGSVM}
\begin{aligned}
\underset{x \in \mathbb{R}^d}{\min}~& \frac{1}{n}\sum_{i = 1}^{n}\phi_i(x) + \gamma\lVert x \rVert_2^2 + \nu\lVert y \rVert_1 \\
\qquad s.t.~~&Bx = y.
\end{aligned}
\end{equation}
Denoting $f_2(x) = \frac{1}{n}\sum_{i = 1}^{n}\phi_i(x) + \gamma\lVert x \rVert_2^2 $, $f_1(x) = \lVert x \rVert_1$
and the approximated augmented Lagrangian as follows,
\begin{equation}
\begin{aligned}
\hat{\mathcal{L}}_{\tilde{\beta},k}(x,y,\lambda)
& = f_2(x_k) + \big< \nabla f_2^{[i_k]}(x_k),x - x_k \big> + f_1(y) \\
& - \big< \lambda,Bx - y \big> + \frac{\tilde{\beta}}{2}\lVert Bx - y \rVert_2^2 + \frac{\lVert x - x_k\rVert_2^2}{2\zeta_{k + 1}},
\end{aligned}
\end{equation}
the STOC-ADMM is formulated as
\begin{equation}\label{SADMM1}
\left \{
\begin{aligned}
x_{k + 1} & = \underset{x \in \mathbb{R}^d}{\arg\min}~ \hat{\mathcal{L}}_{\tilde{\beta},k}(x,y_k,\lambda_k) \\
y_{k + 1} & = \underset{x \in \mathbb{R}^d}{\arg\min}~ \hat{\mathcal{L}}_{\tilde{\beta},k}(x_{k + 1},y,\lambda_k) \\
\lambda_{k + 1} & = \lambda_k - \tilde{\beta}(Bx_{k + 1} - y_{k + 1}),
\end{aligned}
\right.
\end{equation}
which is equivalent to
\begin{equation}\label{SADMM2}
\left \{
\begin{aligned}
x_{k + 1} & = \Big( \frac{I}{\zeta_{k + 1}} + \tilde{\beta} B^TB\Big)^{-1}\Big[B^T(\tilde{\beta} y_k + \lambda_k) + \frac{x_k }{\zeta_{k + 1}}- \nabla f_2^{[i_k]}(x_k)\Big]  \\
y_{k + 1} & = \mbox{Prox}_{f_1/\tilde{\beta}}\Big(Bx_{k + 1} - \frac{\lambda_{k}}{\tilde{\beta}} \Big)\\
\lambda_{k + 1} & = \lambda_k - \tilde{\beta}(Bx_{k + 1} - y_{k + 1}).
\end{aligned}
\right.
\end{equation}
The key step of SPDFP (Algorithm 2) is given as follows:
\begin{equation}\label{SPDFP}
\left \{
\begin{aligned}
\gamma_k & = \frac{c}{k^{\alpha}} \\
x_{k + \frac{1}{2}} & = x_k -  \gamma_k \nabla f_2^{[i_k]}(x_k) \\
v_{k + 1} & = (I - \mbox{Prox}_{\frac{\gamma_k}{\lambda}f_1})\Big(Bx_{k + \frac{1}{2}} + \big(\frac{k - 1}{k}\big)^{\alpha}(I - \lambda BB^T)v_k\Big) \\
x_{k + 1} & = x_{k + \frac{1}{2}} - \lambda B^Tv_{k + 1}.
\end{aligned}
\right.
\end{equation}
It can be seen that the difference between SPDFP and STOC-ADMM is that SPDFP does not need to solve the linear equation, which may decrease the complexity in each iteration.

In the experiment, we compared STOC-ADMM and SPDFP on dataset 20newsgroups\footnote{http://www.cs.nyu.edu/~roweis/data.html}, which is composed of binary occurrences of $100$ popular words counted from $16,242$ newsgroup postings on the top level of which are four main categories: computer, recreation, science, and talks. As in \cite{STOCADMM}, we split $80\%$ of the postings for training and $20\%$ for testing and used the one-versus-rest scheme for this multi-class classification task.  The parameters for STOC-ADMM are exactly the same as in \cite{STOCADMM}. For SPDFP, we set $\lambda = 0.02$ and $\gamma_k = 2/k^{0.55}$ to ensure the best performance. All the other settings in the model are the same for both algorithms. We give the plot of testing accuracy averaged over 10 independent repetitions both as a function of epoch and as a function of time. From Figure \textbf{\ref{GGSVMfg}}, it can be seen that the training accuracy of the SPDFP is as good as that of STOC-ADMM, while SPDFP achieves the highest accuracy much quicker.

\begin{figure}[!htbp]
\centering
\includegraphics[width=0.45\columnwidth]{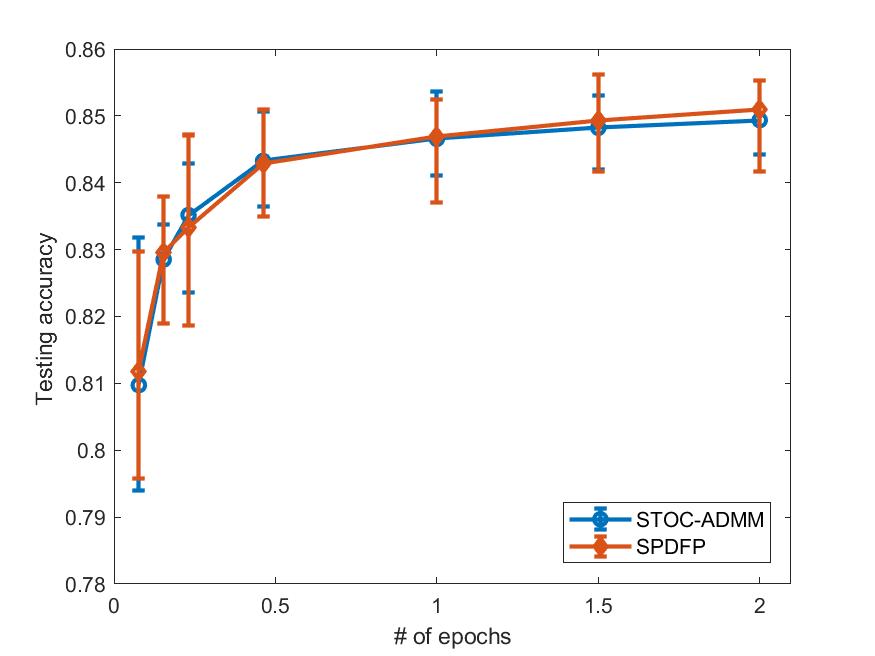}%
\includegraphics[width=0.45\columnwidth]{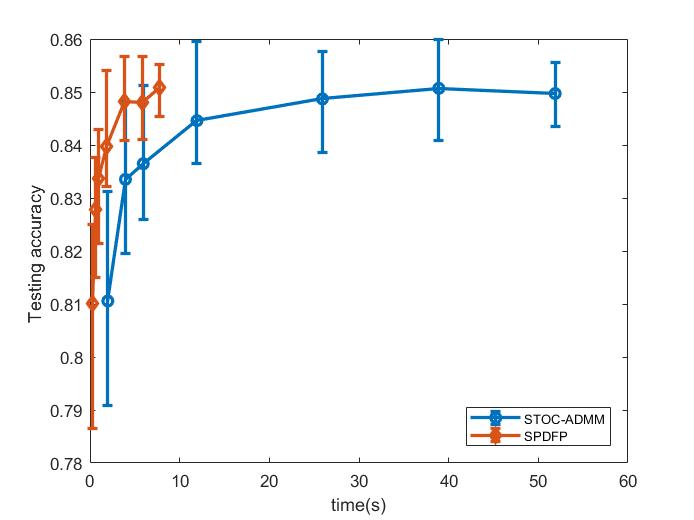}
\caption{Testing accuracy averaged over 10 independent repetitions as a function of epochs (left) and computation time (right) for multi-class classification. Dataset: 20newsgroup}
\label{GGSVMfg}
\end{figure}

\subsection{Graph guide logistic regression.}
The third example that we performed experiments on was the graph guide logistic regression model \cite{Kim}:
\begin{equation}\label{GGLGT1}
\underset{x \in \mathbb{R}^d}{\min}\frac{1}{n}\sum_{i = 1}^{n}\phi_i(x) + \nu_1\lVert Bx \rVert_1,
\end{equation}
where $\phi_i(x) = \log(1 + \exp(-b_is_i^Tx))$. As in \cite{STOCADMM,SVRGADMM}, we used sparse inverse covariance selection \cite{Banerjee2008model} to obtain the graph matrix $G$ and $B = [G;I]$, where $I$ denotes identity. We also added a $l_2$ term to make the problem meet our setting, i.e., we solved the following problem:
\begin{equation}\label{GGLGT2}
\underset{x \in \mathbb{R}^d}{\min}\frac{1}{n}\sum_{i = 1}^{n}\phi_i(x) + \nu_1\lVert x \rVert_2^2 + \nu_2\lVert Bx \rVert_1.
\end{equation}
The following are the details of the example:
\begin{itemize}
  \item We used two datasets, \textbf{a9a} ($54$ features  and $581012$ samples)  and \textbf{covtype} ($123$ features and $32561$ samples), from \textbf{LIBSVM} \cite{CC01a}.
  \item We used half of the set for training and half for testing.
  \item The ground truth of (\textbf{\ref{GGLGT2}}) was derived by running PDFP for $10,000$ iterations.
  \item We compared SPDFP with SCAS-ADMM \cite{SCASADMM}, OPG-ADMM \cite{RDAOPGADMM}, and PDFP \cite{PDFP}.
  \item The mini-batch version of different algorithms has been given; the batch size was 200 for \textbf{a9a} and 1,000 for \textbf{covtype}.
  \item All the algorithms were run 10 times and the average reported.
\end{itemize}%

\begin{figure}[!htbp]
\begin{center}
\subfigure[a9a]{
\centering
\includegraphics[width=0.45\columnwidth]{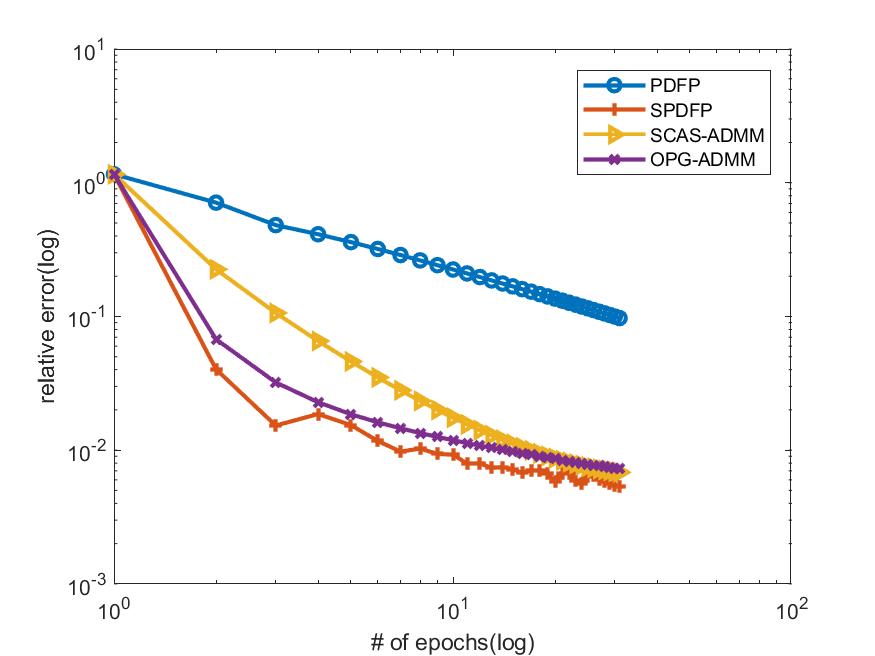}
}
\subfigure[covtype]{
\centering
\includegraphics[width=0.45\columnwidth]{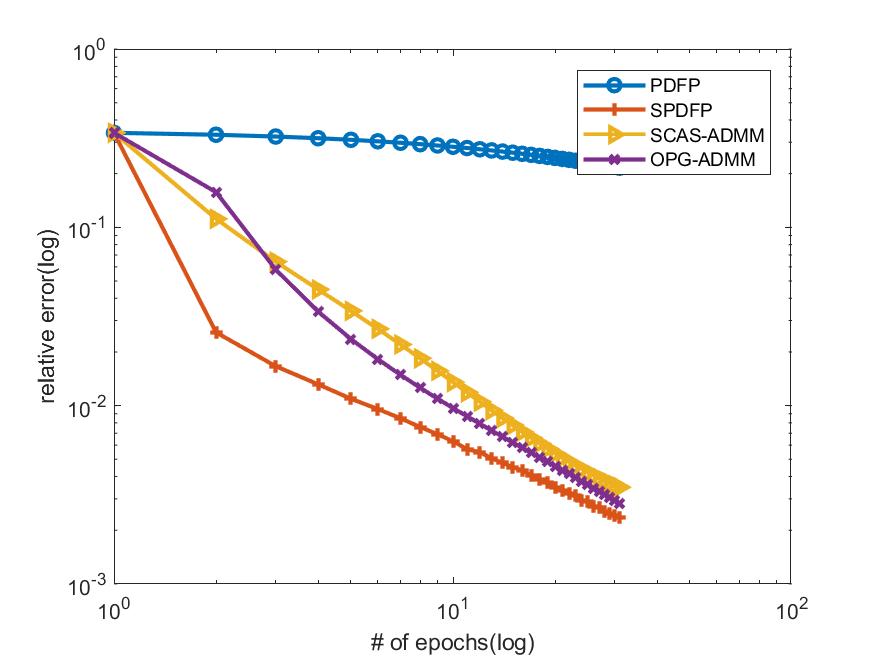}
}
\end{center}
\caption{Averaged relative error of objective value vs epochs over 10 independent repetitions}
\label{GGlgt1}
\end{figure}

\begin{figure}[!htbp]
\begin{center}
\subfigure[a9a]{
\centering
\includegraphics[width=0.45\columnwidth]{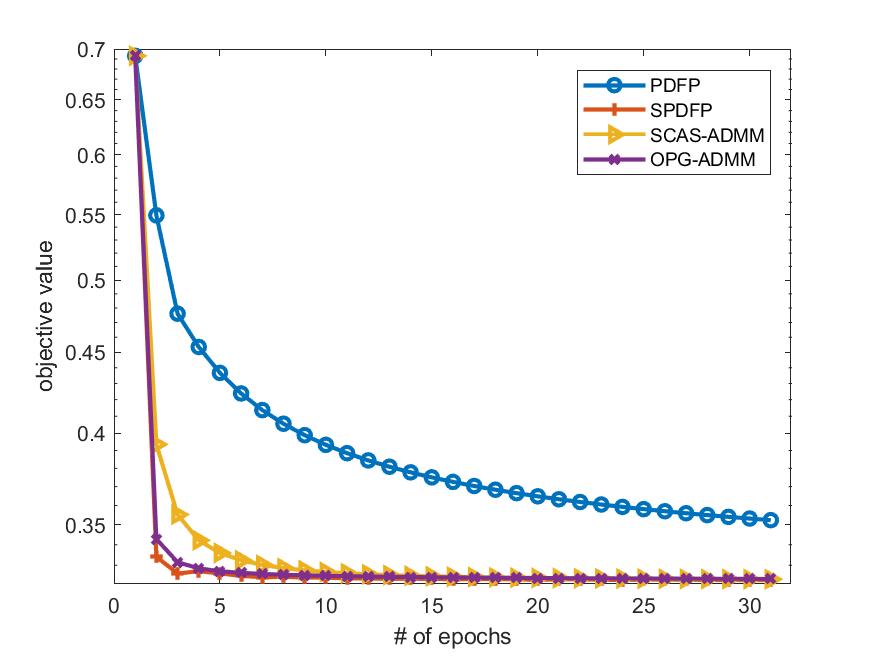}
}
\subfigure[covtype]{
\centering
\includegraphics[width=0.45\columnwidth]{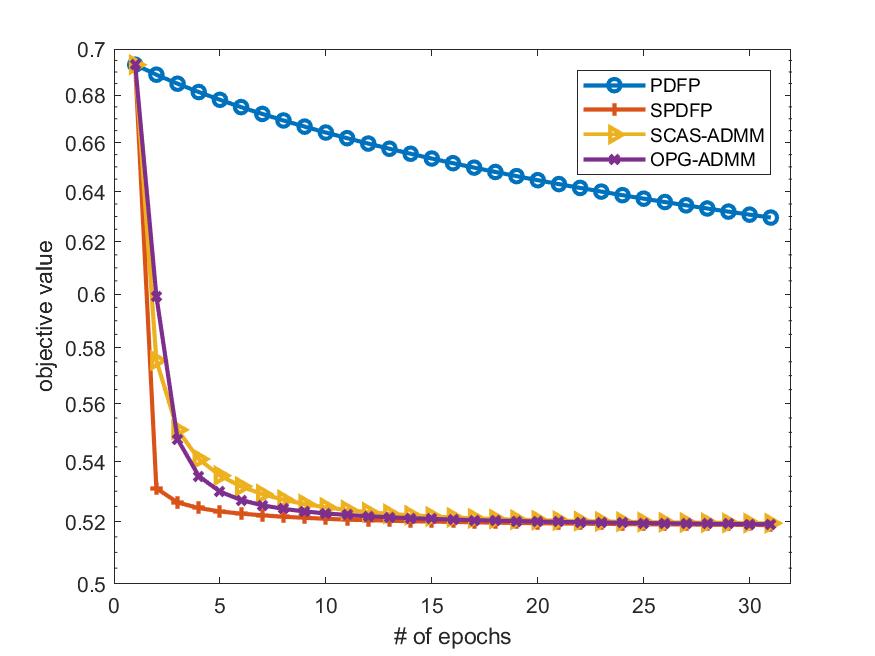}
}
\end{center}
\caption{Averaged objective value vs epochs over 10 independent repetitions}
\label{GGlgt2}
\end{figure}

\begin{figure}[!htbp]
\begin{center}
\subfigure[a9a]{
\centering
\includegraphics[width=0.45\columnwidth]{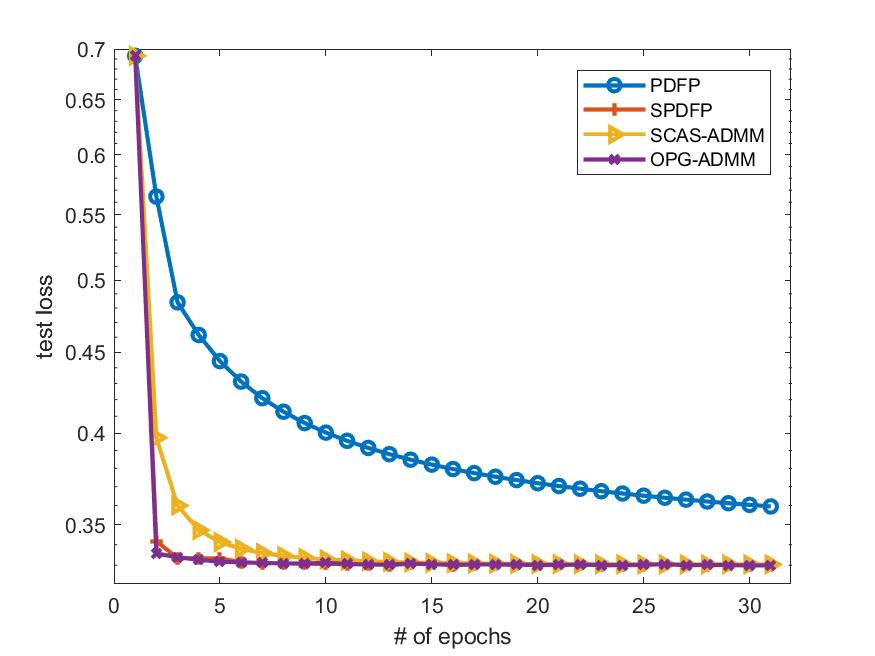}
}
\subfigure[covtype]{
\centering
\includegraphics[width=0.45\columnwidth]{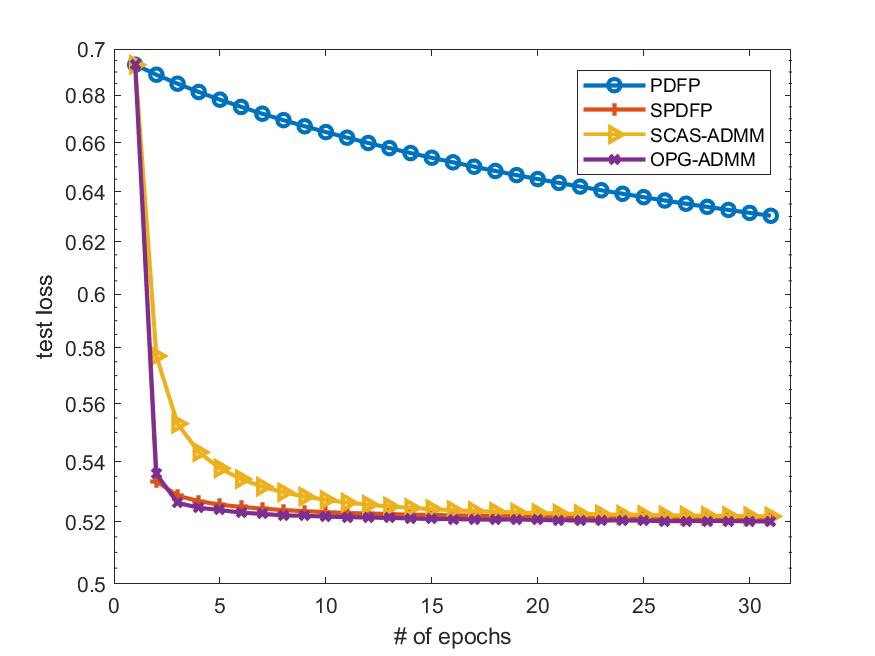}
}
\end{center}
\caption{Averaged testing loss vs epochs over 10 independent repetitions}
\label{GGlgt3}
\end{figure}

\begin{figure}[!htbp]
\begin{center}
\subfigure[a9a]{
\centering
\includegraphics[width=0.45\columnwidth]{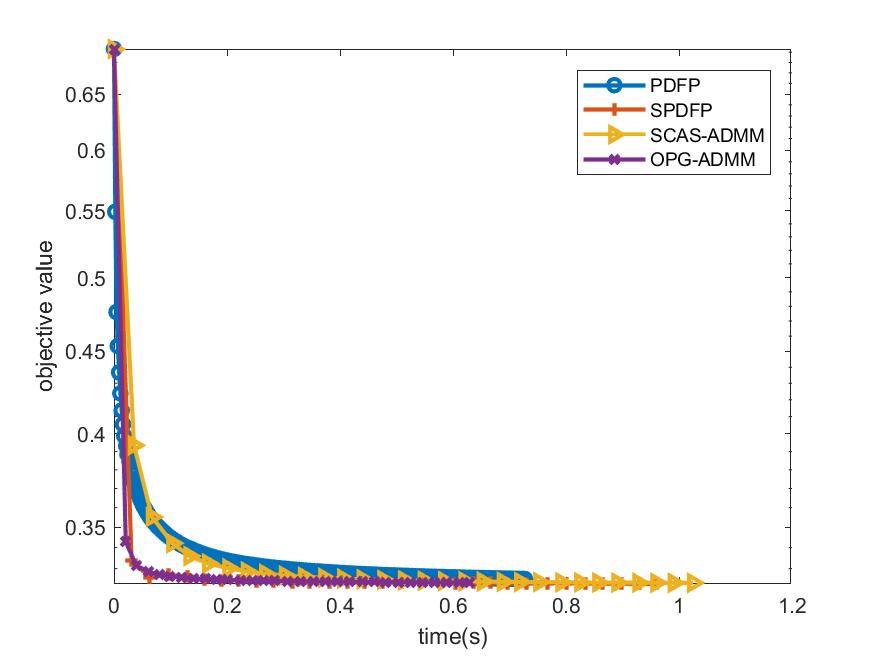}
}
\subfigure[covtype]{
\centering
\includegraphics[width=0.45\columnwidth]{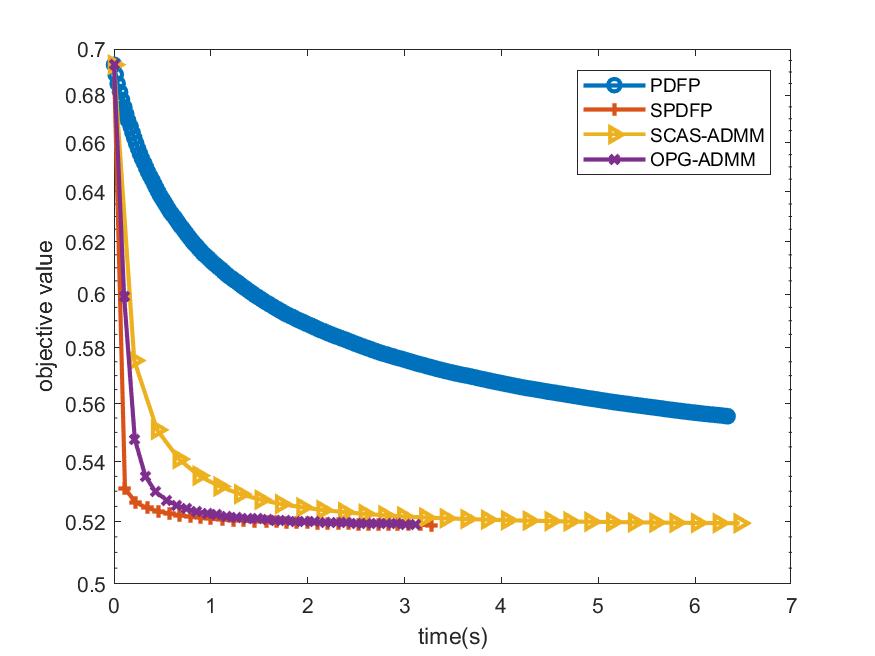}
}
\end{center}
\caption{Averaged objective value vs time over 10 independent repetitions}
\label{GGlgt4}
\end{figure}
Figure \textbf{\ref{GGlgt1}} gives the log-log plot of relative error of the function value and epochs. It can be seen that the performance of SPDFP is as good as that of traditional algorithms, while the performance of PDFP is not good in this large-scale problem. Figures \textbf{\ref{GGlgt2}} and \textbf{\ref{GGlgt3}}, which give the objective value and testing loss of the algorithms, also verify this fact. Figure \textbf{\ref{GGlgt4}} gives the plot of objective value versus time. It can be seen that stochastic algorithms perform much better than deterministic algorithms (we plot $300$ iterations for PDFP in Figure \textbf{\ref{GGlgt4}}, i.e. 300 epochs). \\
We also note that in the convergence analysis, we added an $l_2$ term to make the problem strongly convex. However, if we eliminate the $l_2$ term, i.e., set $\nu_1 = 0$, the SPDFP exhibits similar performance in this example; see Figure \textbf{\ref{GGlgt55} - \ref{GGlgt7}}.

\begin{figure}[htp]
\begin{center}
\subfigure[a9a]{
\centering
\includegraphics[width=0.45\columnwidth]{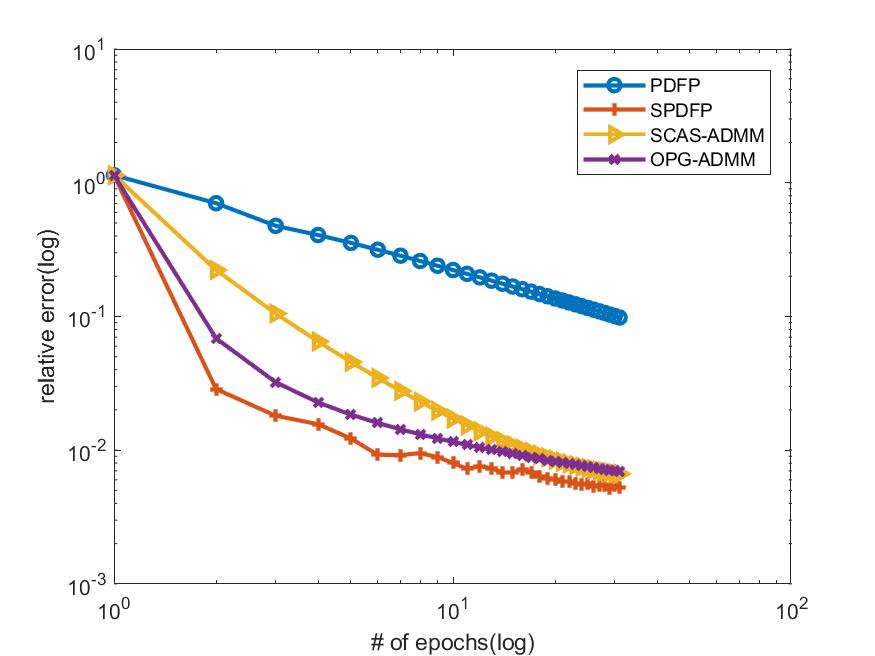}
}
\subfigure[covtype]{
\centering
\includegraphics[width=0.45\columnwidth]{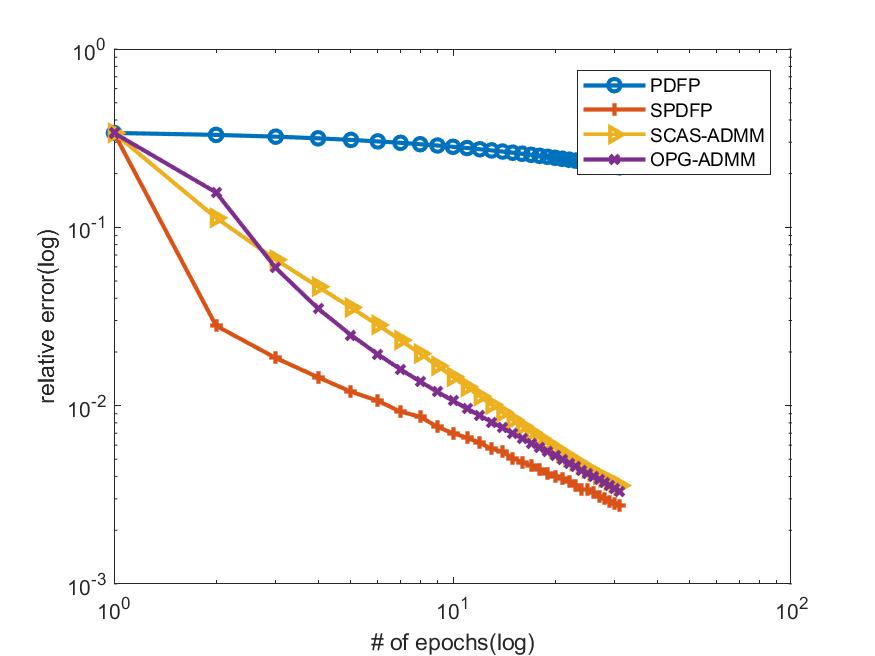}
}
\end{center}
\caption{Averaged relative error of objective value vs epochs over 10 independent repetitions (without $l_2$ term)}
\label{GGlgt55}
\end{figure}

\begin{figure}[htp]
\begin{center}
\subfigure[a9a]{
\centering
\includegraphics[width=0.45\columnwidth]{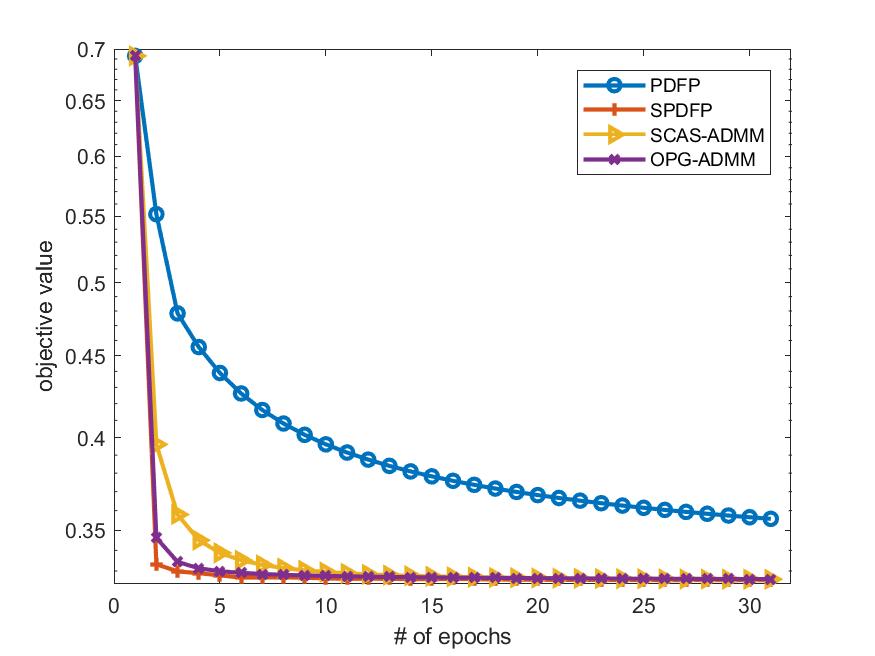}
}
\subfigure[covtype]{
\centering
\includegraphics[width=0.45\columnwidth]{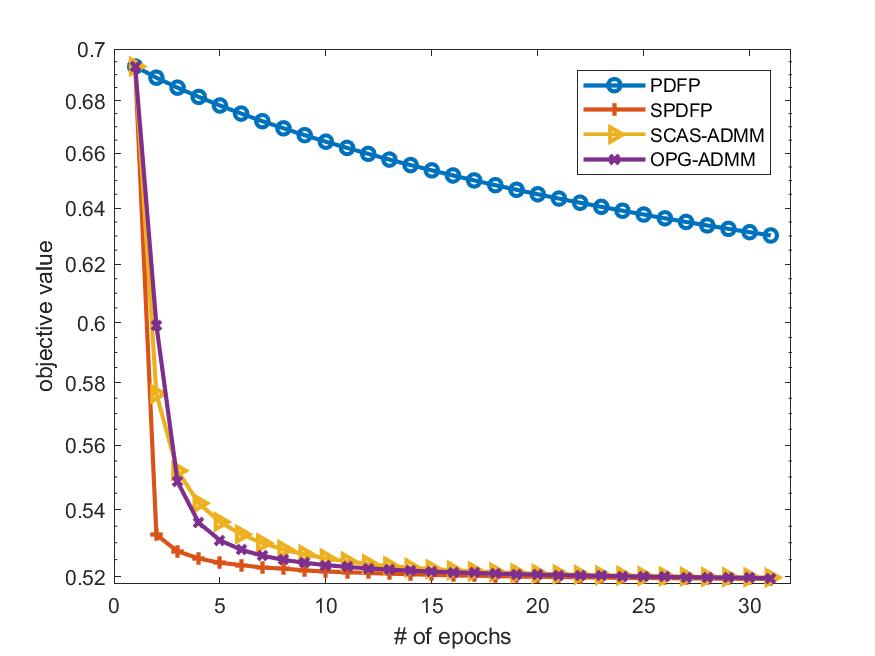}
}
\end{center}
\caption{Averaged objective value vs epochs over 10 independent repetitions (without $l_2$ term)}
\label{GGlgt5}
\end{figure}

\begin{figure}[htp]
\begin{center}
\subfigure[a9a]{
\centering
\includegraphics[width=0.45\columnwidth]{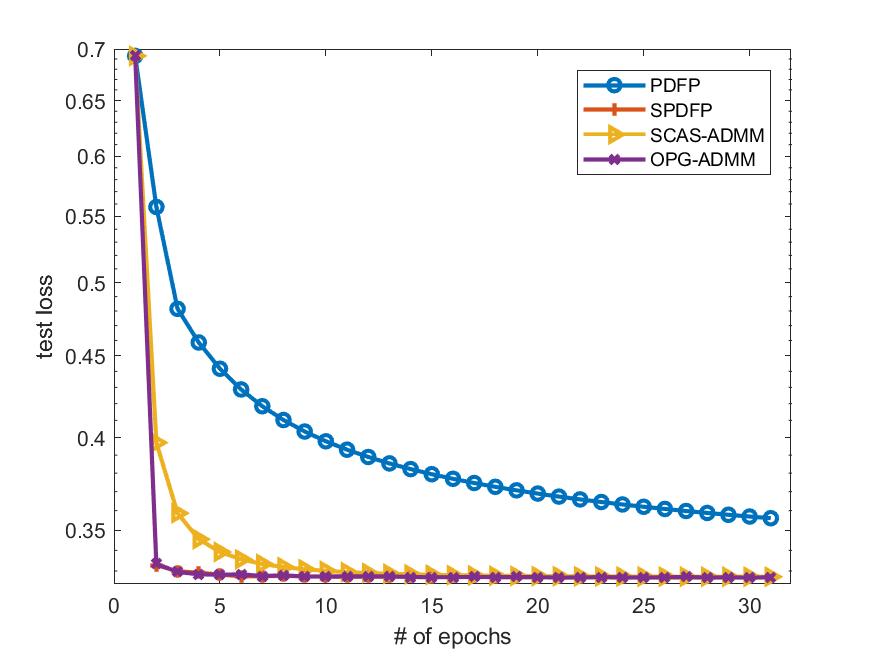}
}
\subfigure[covtype]{
\centering
\includegraphics[width=0.45\columnwidth]{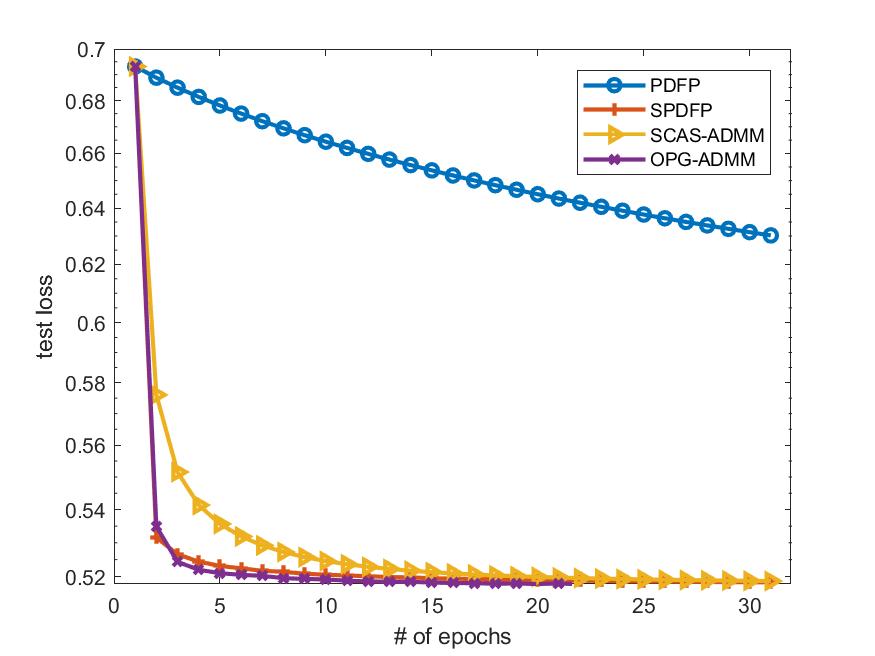}
}
\end{center}
\caption{Averaged testing loss vs epochs over 10 independent repetitions (without $l_2$ term)}
\label{GGlgt6}
\end{figure}

\begin{figure}[htp]
\begin{center}
\subfigure[a9a]{
\centering
\includegraphics[width=0.45\columnwidth]{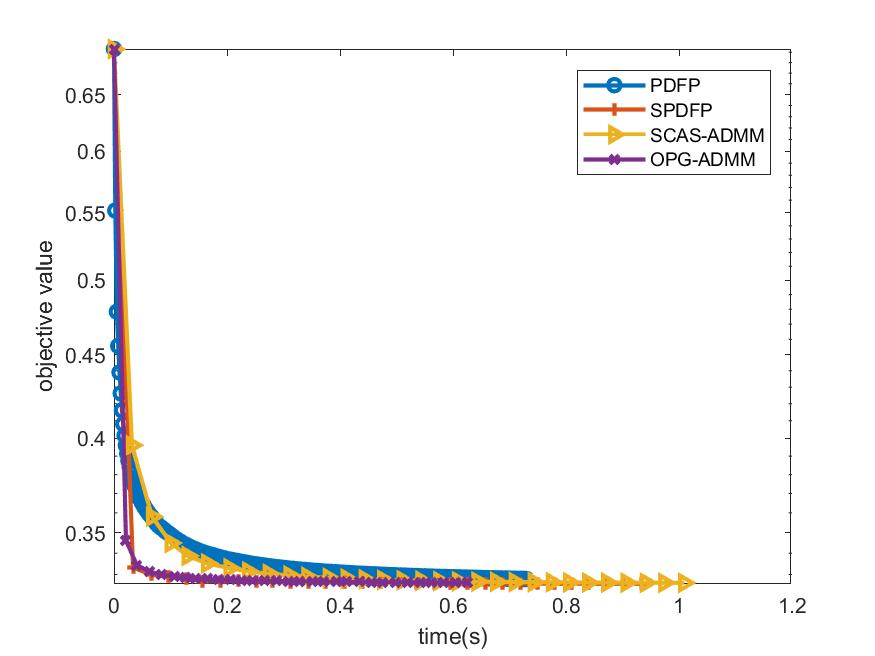}
}
\subfigure[covtype]{
\centering
\includegraphics[width=0.45\columnwidth]{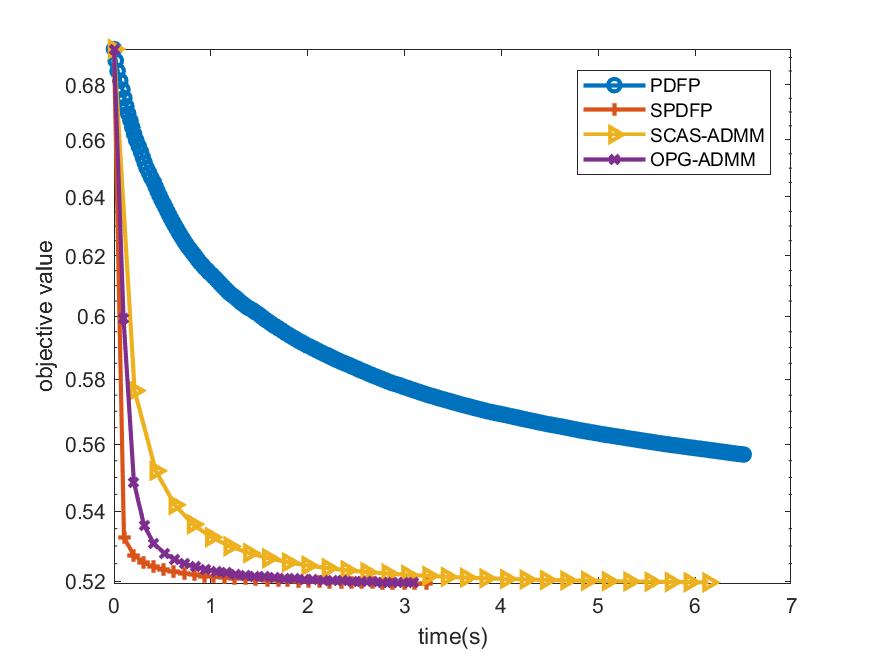}
}
\end{center}
\caption{Averaged objective value vs time over 10 independent repetitions (without $l_2$ term)}
\label{GGlgt7}
\end{figure}

\section{Conclusions.}
In this paper, we propose the stochastic primal dual fixed-point algorithm for solving the problem considered in $\textbf{(\ref{pb1})}$. Based on boundness and non-boundness assumptions of gradient of $f_2(x)$, we give the convergence and convergence rate of SPDFP. Under mild conditions, SPDFP can achieve a $\mathcal{O}(k^{-\alpha})$ rate, where $k$ is iteration number and $\alpha \in (0,1]$. The efficiency of SPDFP is confirmed through examples on fussed lasso, graph guide SVM, and graph guide logistic regression for some real-world datasets.
\section{Appendix.}
Here, we give the details of the aforementioned lemmas. First, we give a lemma that will be used in the proof of Lemma \textbf{\ref{lm3}}.
\begin{lemma}\label{APlm1}
Letting $r > 0$ and $h(x) = r f_0(x/r),x \in \mathbb{R}^d$, then, for all$ ~ y \in \mathbb{R}^d$,
\begin{equation*}
\mathrm{Prox}_h(y) = r \mathrm{Prox}_{r^{-1}f_0}(y/r).
\end{equation*}
\end{lemma}
\begin{proof}
The assertion can be proved by using the definition of $\mathrm{Prox}_{f_0}(\cdot)$ and change of variables.
\end{proof}
\newpage
\subsection{Proof of Lemma \textbf{\ref{lm3}}}

\begin{proof}
By the first optimality condition of problem \textbf{(\ref{pb1})}, we have
\begin{equation}\label{apeq}
\begin{aligned}
x^{*}
& = \underset{x \in \mathbb{R}^d}{\arg\min} (f_1 \circ B)(x) + f_2(x)  \\
& \Leftrightarrow 0 \in -\gamma_k \nabla f_2(x^*) - \gamma_k \partial (f_1 \circ B )(x^*) \\
& \Leftrightarrow x^* \in x^* -\gamma_k \nabla f_2(x^*) - \gamma_k B^T \partial f_1(Bx^*) \\
& \Leftrightarrow x^* \in x^* -\gamma_k \nabla f_2(x^*) - \lambda \big(B^T \circ \frac{\gamma_k}{\lambda} \partial f_1(Bx^*).
\end{aligned}
\end{equation}

\noindent
Letting
\begin{equation}\label{apeq1}
v^* \in \partial f_1(Bx^*),
\end{equation}
then equation \textbf{(\ref{apeq})} can be rewritten as
\begin{equation}\label{apeq5}
x^* = x^* - \gamma_k \nabla f_2(x^*) - \gamma_k B^T v^* .
\end{equation}
From equation \textbf{(\ref{apeq1})}, we also have
\begin{equation}\label{apeq3}
\frac{\gamma_k}{\lambda}v^* \in \partial \frac{\gamma_k}{\lambda}f_1(Bx^*),
\end{equation}
which means that
\begin{equation}\label{apeq2}
\begin{aligned}
Bx^* & = \mathrm{Prox}_{\frac{\gamma_k}{\lambda}f_1}(Bx^* + \frac{\gamma_k}{\lambda}v^*) \\
& \Leftrightarrow (Bx^* + \frac{\gamma_k}{\lambda}v^*) - \frac{\gamma_k}{\lambda}v^* = \mathrm{Prox}_{\frac{\gamma_k}{\lambda}f_1}(Bx^* + \frac{\gamma_k}{\lambda}v^*) \\
& \Leftrightarrow \frac{\gamma_k}{\lambda}v^* = (I - \mathrm{Prox}_{\frac{\gamma_k}{\lambda}f_1})(Bx^* + \frac{\gamma_k}{\lambda}v^*) \\
& \Leftrightarrow v^* = \frac{\lambda}{\gamma_k}(I - \mathrm{Prox}_{\frac{\gamma_k}{\lambda}f_1})(Bx^* + \frac{\gamma_k}{\lambda}v^*) \\
& \Leftrightarrow v^* = \frac{\lambda}{\gamma_k}Bx^* + v^* - \frac{\lambda}{\gamma_k}\mathrm{Prox}_{\frac{\gamma_k}{\lambda}f_1}(Bx^* + \frac{\gamma_k}{\lambda}v^*) \\
& \Leftrightarrow v^* = \frac{\lambda}{\gamma_k}Bx^* + v^* - \frac{\lambda}{\gamma_k}\mathrm{Prox}_{(\frac{\lambda}{\gamma_k})^{-1}f_1}\Big(\frac{\frac{\lambda}{\gamma_k}Bx^* + v^*}{\frac{\lambda}{\gamma_k}}\Big) \\
& \Leftrightarrow v^* = \frac{\lambda}{\gamma_k}Bx^* + v^* - \mathrm{Prox}_{h^k}\Big(\frac{\lambda}{\gamma_k}Bx^* + v^*\Big) \\
& \Leftrightarrow v^* = (I- \mathrm{Prox}_{h^k})\Big(\frac{\lambda}{\gamma_k}Bx^* + v^*\Big)
\end{aligned}
,
\end{equation}
where, in the second-to-last equality, we let $h^k(x) =  \frac{\lambda}{\gamma_k}f_1(\frac{x}{\frac{\lambda}{\gamma_k}})
= \frac{\lambda}{\gamma_k}f_1(\frac{\gamma_k}{\lambda}x)$ and using Lemma \textbf{\ref{APlm1}}. \\
\noindent Inserting Eq. \textbf{(\ref{apeq5})} into the last equality of Eq. \textbf{(\ref{apeq2})}, we have
\begin{equation}\label{sumup}
\begin{aligned}
v^* & = (I - \mathrm{Prox}_{h^k})\Big(\frac{\lambda}{\gamma_k}B(x^* - \gamma_k \nabla f_2(x^*)) + (I - \lambda BB^T)v^* \Big)  \\
\Leftrightarrow v^* & = \frac{\lambda}{\gamma_k}(I - \mathrm{Prox}_{\frac{\gamma_k}{\lambda}f_1})\Big(B(x^* - \gamma_k \nabla f_2(x^*)) + (I - \lambda BB^T)\frac{\gamma_k}{\lambda}v^* \Big).
\end{aligned}
\end{equation}
Combining \textbf{(\ref{apeq5})}, \textbf{(\ref{apeq2})} and \textbf{(\ref{sumup})}, we obtain
\begin{equation*}
\left\{
\begin{aligned}
v^* & = (I - \mathrm{Prox}_{h^k})\Big(\frac{\lambda}{\gamma_k}B(x^* - \gamma_k \nabla f_2(x^*)) + (I - \lambda BB^T)v^* \Big) \\
& = T_0^{(k)}(x^*,v^*) \\
x^* & = x^* - \gamma_k \nabla f_2(x^*) - \gamma_k B^T T_0^{(k)}(x^*,v^*).
\end{aligned}
\right.
\end{equation*}
The converse can be similarly verified.
This completes the proof. \qquad
\end{proof}

\subsection{Proof of Lemma \textbf{\ref{cglmeq1}}.}
\begin{proof}
Letting $(x^*,v^*)$ be that in Lemma \textbf{\ref{lm3}} and $(x_k,v_k)$ be the iterate in Algorithm 1, $T_1^{(k)}(\cdot) = (I - \mathrm{Prox}_{h^k})(\cdot)$, where $h^k(x) =  \frac{\lambda}{\gamma_k}f_1(\frac{x}{\frac{\lambda}{\gamma_k}}) = \frac{\lambda}{\gamma_k}f_1(\frac{\gamma_k}{\lambda}x)$.
 We denote
\begin{equation*}
\begin{aligned}
\varphi_{1,i}^{(k)}(x,y) & = \frac{\lambda}{\gamma_k}B(x - \gamma_k\nabla f_2^{[i]}(x))+ (I - \lambda B B^T)y \\
& = \frac{\lambda}{\gamma_k}Bg_{k,i}^{(1)}(x) + My \\
\varphi_2^{(k)}(x,y) & = \frac{\lambda}{\gamma_k}B(x - \gamma_k\nabla f_2(x))+ (I - \lambda B B^T)y \\
& = \frac{\lambda}{\gamma_k}Bg_k^{(2)}(x) + My \\
\end{aligned}
\end{equation*}
Here, $g_{k,i}^{(1)}(x) = x - \gamma_k\nabla f_2^{[i]}(x)$ and $g_{k}^{(2)}(x) = x - \gamma_k\nabla f_2(x)$, $M = I - \lambda B B^T$.

\begin{enumerate}[label=\textbf{\roman*}),leftmargin= *]
\item  Estimation of $\lVert v_{k + 1} - v^* \rVert_2^2$:
\begin{equation}\label{eq2}
\begin{aligned}
\lVert v_{k + 1} - v^* \rVert_2^2
& = \lVert T_1^{(k)}(\varphi_{1,i_k}^{(k)}(x_k,v_k)) - v^* \rVert_2^2 \\
& = \lVert T_1^{(k)}(\varphi_{1,i_k}^{(k)}(x_k,v_k)) - T_1^{(k)}(\varphi_2^{(k)}(x^*,v^*)) \rVert_2^2 \\
& \leq \big< T_1^{(k)}(\varphi_{1,i_k}^{(k)}(x_k,v_k)) - T_1^{(k)}(\varphi_2^{(k)}(x^*,v^*)),\varphi_{1,i_k}^{(k)}(x_k,v_k) - \varphi_2^{(k)}(x^*,v^*) \big> \\
& = \frac{\lambda}{\gamma_k}\big< T_1^{(k)}(\varphi_{1,i_k}^{(k)}(x_k,v_k)) - T_1^{(k)}(\varphi_2^{(k)}(x^*,v^*)),B(g_{k,i_k}^{(1)}(x_k) - g_k^{(2)}(x^*)) \big> \\
& + \big< T_1^{(k)}(\varphi_{1,i_k}^{(k)}(x_k,v_k)) - T_1^{(k)}(\varphi_2^{(k)}(x^*,v^*)), M(v_k - v^*) \big>.
\end{aligned}
\end{equation}
The second equality follows from Eq. \textbf{(\ref{eq1})} and the inequality follows from the firm non-expansiveness of $T_1^{(k)}$ (see definition \textbf{\ref{lm1}}).

Here, and in what follows, for convenience, we denote $T_1^{(k)}(\varphi_{1,i_k}^{(k)}) = T_1^{(k)}(\varphi_{1,i_k}^{(k)}(x_k,v_k))$ and
$T_1^{(k)}(\varphi_2^{(k)}) = T_1^{(k)}(\varphi_2^{(k)}(x^*,v^*))$.

\item Estimation of $\lVert x_{k + 1} - x^* \rVert_2^2:$
\begin{equation}\label{eq3}
\begin{aligned}
&\lVert x_{k + 1} - x^* \rVert_2^2 \\
& = \lVert x_k - \gamma_k \nabla f_2^{[i_k]}(x_k) - \gamma_k B^T \circ T_1^{(k)}(\varphi_{1,i_k}^{(k)}) - \big(x^* - \gamma_k \nabla f_2(x^*) - \gamma_k B^T \circ T_1^{(k)}(\varphi_2^{(k)}) \big)  \rVert_2^2 \\
& = \lVert g_{k,i_k}^{(1)}(x_k) - \gamma_k B^T \circ T_1^{(k)}(\varphi_{1,i_k}^{(k)}) - \big(g_k^{(2)}(x^*) - \gamma_k B^T \circ T_1^{(k)}(\varphi_2^{(k)}) \big) \rVert_2^2 \\
& = \lVert g_{k,i_k}^{(1)}(x_k) - g_k^{(2)}(x^*) - \gamma_k B^T \circ \big(T_1^{(k)}(\varphi_{1,i_k}^{(k)}) - T_1^{(k)}(\varphi_2^{(k)}) \big)\rVert_2^2 \\
& = \lVert g_{k,i_k}^{(1)}(x_k)  - g_k^{(2)}(x^*) \rVert_2^2  - 2\gamma_k \big< B^T \circ \big( T_1^{(k)}(\varphi_{1,i_k}^{(k)}) - T_1^{(k)}(\varphi_2^{(k)})\big), g_{k,i_k}^{(1)}(x_k) - g_k^{(2)}(x^*) \big> \\
& + \frac{\gamma_k^2}{\lambda^2}\lVert \lambda B^T \circ ( T_1^{(k)}(\varphi_{1,i_k}^{(k)}) - T_1^{(k)}(\varphi_2^{(k)}) ) \rVert_2^2 \\
& = \lVert g_{k,i_k}^{(1)}(x_k)  - g_k^{(2)}(x^*) \rVert_2^2  - 2\gamma_k \big< B^T \circ \big( T_1^{(k)}(\varphi_{1,i_k}^{(k)}) - T_1^{(k)}(\varphi_2^{(k)})\big), g_{k,i_k}^{(1)}(x_k) - g_k^{(2)}(x^*) \big> \\
& - \frac{\gamma_k^2}{\lambda}\lVert ( T_1^{(k)}(\varphi_{1,i_k}^{(k)}) - T_1^{(k)}(\varphi_2^{(k)}) ) \rVert_M^2 + \frac{\gamma_k^2}{\lambda}\lVert ( T_1^{(k)}(\varphi_{1,i_k}^{(k)}) - T_1^{(k)}(\varphi_2^{(k)}) ) \rVert_2^2.
\end{aligned}
\end{equation}
The first equality follows from Eq. \textbf{(\ref{eq1})}. In the last equality, we use the definition $M = I - \lambda BB^T$ and $\lVert y \rVert_M = \sqrt{<y,My>}$.

\item From \textbf{(\ref{eq2})} and \textbf{(\ref{eq3})}, we have
\begin{equation}\label{eq4}
\begin{aligned}
& \lVert x_{k + 1} - x^* \rVert_2^2 + \frac{\gamma_{k + 1}^2}{\lambda}\lVert v_{k + 1} - v^* \rVert_2^2 \\
& =  \lVert g_{k,i_k}^{(1)}(x_k)  - g_k^{(2)}(x^*) \rVert_2^2  - 2\gamma_k \big< B^T \circ \big( T_1^{(k)}(\varphi_{1,i_k}^{(k)}) - T_1^{(k)}(\varphi_2^{(k)})\big), g_{k,i_k}^{(1)}(x_k) - g_k^{(2)}(x^*) \big> \\
& - \frac{\gamma_{k}^2}{\lambda}\lVert ( T_1^{(k)}(\varphi_{1,i_k}^{(k)}) - T_1^{(k)}(\varphi_2^{(k)}) ) \rVert_M^2 + \frac{\gamma_{k}^2}{\lambda}\lVert ( T_1^{(k)}(\varphi_{1,i_k}^{(k)}) - T_1^{(k)}(\varphi_2^{(k)}) ) \rVert_2^2 \\
& + \frac{\gamma_{k + 1}^2}{\lambda}\lVert ( T_1(\varphi_{1,i_k}^{(k)}) - T_1^{(k)}(\varphi_2^{(k)}) ) \rVert_2^2 \\
& \leq \lVert g_{k,i_k}^{(1)}(x_k)  - g_k^{(2)}(x^*) \rVert_2^2  - 2\gamma_k \big< B^T \circ \big( T_1^{(k)}(\varphi_{1,i_k}^{(k)}) - T_1^{(k)}(\varphi_2^{(k)})\big), g_{k,i_k}^{(1)}(x_k) - g_k^{(2)}(x^*) \big> \\
& - \frac{\gamma_{k}^2}{\lambda}\lVert ( T_1^{(k)}(\varphi_{1,i_k}^{(k)}) - T_1^{(k)}(\varphi_2^{(k)}) ) \rVert_M^2 + 2\frac{\gamma_{k}^2}{\lambda}\lVert ( T_1^{(k)}(\varphi_{1,i_k}^{(k)}) - T_1^{(k)}(\varphi_2^{(k)}) ) \rVert_2^2 \\
& \leq \lVert g_{k,i_k}^{(1)}(x_k)  - g_k^{(2)}(x^*) \rVert_2^2  - \underline{2\gamma_k \big< B^T \circ \big( T_1^{(k)}(\varphi_{1,i_k}^{(k)}) - T_1^{(k)}(\varphi_2^{(k)})\big), g_{k,i_k}^{(1)}(x_k) - g_k^{(2)}(x^*) \big>} \\
& - \frac{\gamma_{k}^2}{\lambda}\lVert ( T_1^{(k)}(\varphi_{1,i_k}^{(k)}) - T_1^{(k)}(\varphi_2^{(k)}) ) \rVert_M^2 \\
& + \underline{2\frac{\gamma_{k}^2}{\lambda} \frac{\lambda}{\gamma_k}\big< T_1^{(k)}(\varphi_{1,i_k}^{(k)}) - T_1^{(k)}(\varphi_2^{(k)}),B(g_{k,i_k}^{(1)}(x_k) - g_k^{(2)}(x^*)) \big>} \\
& + 2\frac{\gamma_{k}^2}{\lambda} \big< T_1^{(k)}(\varphi_{1,i_k}^{(k)}) - T_1^{(k)}(\varphi_2^{(k)}), M(v_k - v^*) \big>\\
& = \lVert g_{k,i_k}^{(1)}(x_k)  - g_k^{(2)}(x^*) \rVert_2^2  - \underline{2\gamma_k \big< B^T \circ \big( T_1^{(k)}(\varphi_1^{(k)}) - T_1^{(k)}(\varphi_2^{(k)})\big), g_{k,i_k}^{(1)}(x_k) - g_k^{(2)}(x^*) \big>} \\
& - \frac{\gamma_{k}^2}{\lambda}\lVert ( T_1^{(k)}(\varphi_{1,i_k}^{(k)}) - T_1^{(k)}(\varphi_2^{(k)}) ) \rVert_M^2
+ \underline{2\gamma_{k}\big< T_1^{(k)}(\varphi_1^{(k)}) - T_1^{(k)}(\varphi_2^{(k)}),B(g_{k,i_k}^{(1)}(x_k) - g_k^{(2)}(x^*)) \big>} \\
& + 2\frac{\gamma_{k}^2}{\lambda} \big< T_1^{(k)}(\varphi_{1,i_k}^{(k)}) - T_1^{(k)}(\varphi_2^{(k)}), M(v_k - v^*) \big>\\
\end{aligned}
\end{equation}
\begin{equation*}
\begin{aligned}
& = \lVert g_{k,i_k}^{(1)}(x_k)  - g_k^{(2)}(x^*) \rVert_2^2 + 2\frac{\gamma_{k}^2}{\lambda}\big< T_1^{(k)}(\varphi_{1,i_k}^{(k)}) - T_1^{(k)}(\varphi_2^{(k)}), M(v_k - v^*) \big>  \\
& \qquad\qquad\qquad - \frac{\gamma_{k}^2}{\lambda}\lVert ( T_1^{(k)}(\varphi_{1,i_k}^{(k)}) - T_1^{(k)}(\varphi_2^{(k)}) ) \rVert_M^2 \\
& = \lVert g_{k,i_k}^{(1)}(x_k)  - g_k^{(2)}(x^*) \rVert_2^2 + \frac{\gamma_{k}^2}{\lambda}\lVert v_k - v^* \rVert_M^2 - \frac{\gamma_{k}^2}{\lambda}\lVert ( T_1^{(k)}(\varphi_{1,i_k}^{(k)}) - T_1^{(k)}(\varphi_2^{(k)}) ) - (v_k - v^*) \rVert_M^2 \\
& \leq \lVert g_{k,i_k}^{(1)}(x_k)  - g_k^{(2)}(x^*) \rVert_2^2 + \frac{\gamma_{k}^2}{\lambda}(1 - \lambda \rho_{min}(B B^T))\lVert v_k - v^* \rVert_2^2,
\end{aligned}
\end{equation*}
where the first equality uses Eq. (\textbf{\ref{eq3}}) and the second equality of (\textbf{\ref{eq2}}). The first inequality uses the fact that $\gamma_k$ is decreasing with respect to $k$. The second inequality uses Eq. (\textbf{\ref{eq2}}). The last inequality uses the fact that $0 < \lambda \leq \frac{1}{\rho_{max}(B B^T)} $, which means that $0 \preceq M \preceq (1 - \lambda \rho_{min}(B B^T))$.
\end{enumerate}
 Considering the expectations of both sides of inequality \textbf{(\ref{eq4})}, we obtain
\begin{equation}\label{eq5}
\begin{aligned}
& \mathbb{E}^{(k + 1)}(\lVert x_{k + 1} - x^* \rVert_2^2 + \frac{\gamma_{k+1}^2}{\lambda}\lVert v_{k + 1} - v^* \rVert_2^2  )\\
 & \leq \mathbb{E}^{(k + 1)}(\lVert g_{k,i_k}^{(1)}(x_k)  - g_k^{(2)}(x^*) \rVert_2^2)
+ \frac{\gamma_{k}^2}{\lambda}(1 - \lambda \rho_{min}(B B^T)) \mathbb{E}^{(k)}( \lVert v_k - v^* \rVert_2^2 ) \\
& = \mathbb{E}^{(k + 1)}\Big(\lVert x_k - x^* - \gamma_k(\nabla f_2^{[i_k]}(x_k) - \nabla f_2(x^*)) \rVert_2^2 \Big) + \frac{\gamma_{k}^2}{\lambda}(1 - \lambda \rho_{min}(B B^T)) \mathbb{E}^{(k)}( \lVert v_k - v^* \rVert_2^2 ) \\
& = \mathbb{E}^{(k)}\big(\lVert x_k - x^* \rVert_2^2\big) + \frac{\gamma_{k}^2}{\lambda}(1 - \lambda \rho_{min}(B B^T)) \mathbb{E}^{(k)}( \lVert v_k - v^* \rVert_2^2 ) \\
& - 2\gamma_k \mathbb{E}^{(k)}\big< \nabla f_2(x_k) - \nabla f_2(x^*), x_k - x^*\big> + \gamma_k^2 \mathbb{E}^{(k + 1)}(\lVert \nabla f_2^{[i_k]}(x_k) - \nabla f_2(x^*) \rVert_2^2),
\end{aligned}
\end{equation}
where, in the third term of the last equality, we use the fact that $$\mathbb{E}^{(k + 1)} \big(\nabla f_2^{[i_k]}(x_k)\big) = \mathbb{E}^{(k)}\big(\nabla f_2(x_k)\big).$$ This completes the proof. \qquad
\end{proof}

%

\end{document}